\documentclass[12pt]{amsart}
\usepackage{array,multirow}
\usepackage{pdfsync}
\usepackage[utf8]{inputenc}
\usepackage[T1]{fontenc}
\usepackage{amsthm,amsmath,amssymb}
\usepackage[backend=bibtex,style=numeric,doi=false,url=false]{biblatex}
\bibliography{eigene,andere,arxiv,mathscinet}

\usepackage{hyperref}
\usepackage{cleveref}
\creflabelformat{enumi}{(#2#1#3)}
\crefname{enumi}{}{}
\usepackage{listings}
\usepackage{color}
 \newtheorem{Thm}{Theorem}[section]
 \newtheorem{Lem}[Thm]{Lemma}
 \newtheorem{LemDef}[Thm]{Lemma and Definition}
 \newtheorem{Prop}[Thm]{Proposition}
 \newtheorem{Cor}[Thm]{Corollary}
\theoremstyle{remark}

 \newtheorem{Rem}[Thm]{Remark}
\theoremstyle{definition}
 \newtheorem{Def}[Thm]{Definition}
\numberwithin{equation}{section}

\newcommand\Z{\mathbb Z}
\newcommand\CC{\mathbb C}
\newcommand\PSL{\operatorname{PSL}}
\newcommand\ol[1]{\overline{#1}}

\newcommand\adhit{\triangleright}

\newcommand\inv{^{-1}}
\def\HM#1.#2.#3.#4.{{^{#1}_{#3}\mathcal M^{#2}_{#4}}}

\newcommand\ot{\otimes}

\newcommand\Ind{\operatorname{Ind}}

\newcommand\Irr{\operatorname{Irr}}

\newcommand\CMC[4]{M^{#1}(#2,#3,#4)}
\newcommand\Cl{\operatorname{Cl}}
\newcommand\rCl{\underline{\Cl}}
\newcommand\rP{\underline P}

\newcommand\cp{P}
\newcommand\cpc{\ol P}
\newcommand\rcpc{\underline P}
\newcommand\dcf{double class function}
\newcommand\dcfs{double class functions}
\newcommand\be[2]{\beta_{#1}(#2)}
\newcommand\adams[1]{\psi_{#1}}
\newcommand\CF[1]{\mathbb Q(\zeta_{#1})}
\begin{document}
 \title[]{Higher Frobenius-Schur indicators for Drinfeld doubles of finite groups through characters of centralizers}
 \author{Peter Schauenburg}
 \address{Institut de Math{\'e}matiques de Bourgogne --- UMR 5584 du CNRS\\
 Universit{\'e} Bourgogne Franche-Comté\\
 F-21000 Dijon\\France}
 \email{peter.schauenburg@u-bourgogne.fr}
 \subjclass[2010]{18D10,16T05,20C15}
 \keywords{Hopf algebra, Frobenius-Schur indicator, Drinfeld double}
 \begin{abstract}
   We present a new approach to calculating the higher Frobenius-Schur indicators for the simple modules over the Drinfeld double of a finite group. In contrast to the formula by Kashina-Sommerhäuser-Zhu that involves a sum over all group elements satisfying a certain condition, our formula operates on the level of conjugacy classes and character tables. It can be implemented in the computer algebra system GAP, efficiently enough to deal, on a laptop, with symmetric groups up to $S_{18}$ (providing further evidence that indicators are non-negative in this case) or simple groups of order up to $2\cdot 10^8$. The approach also allows us to test whether all indicators over the double of a given group are rational, without computing them. Among simple groups of order up to about $5\cdot 10^{11}$ an inspection yields exactly one example (of order about $5\cdot 10^9$) where irrational indicators occur.
 \end{abstract}
 \maketitle

\section{Introduction}
\label{sec:introduction}

If the reader has seen Drinfeld doubles of finite groups and Frobenius-Schur indicators, the present paper's title will seem to her to announce an utter triviality: Higher Frobenius-Schur indicators are invariants of (simple) objects of suitable fusion categories. Here, we are dealing with the indicators defined in \cite{MR1808131} for the (simple) modules over a semisimple complex Hopf algebra, and more specifically over the Drinfeld double $D(G)$ of a finite group. Those simple modules are well known to be parametrized by the irreducible characters of the centralizers in $G$ of elements of $G$. It thus seems natural that their indicators should also be described in terms of those irreducible characters of centralizers, and indeed the groundbreaking paper \cite{MR2213320} already gives formulas for the indicators involving the description alluded to above.

So what is new in the present paper?

The formula in \cite{MR2213320} for Frobenius-Schur indicators for (the simple modules over) the Drinfeld double of the complex group algebra of a finite group was studied in  \cite{MR3275646,MR2941565,MR2925444,MR3178056,MR3103664}. Let $V$ be a simple $D(G)$-module described by a couple $(g,\eta)$ where $g\in G$, and  $\eta$ is a character of the centralizer $C_G(g)$. Then the formula in question does describe the $m$-th Frobenius-Schur indicator $\nu_m(V)$ in terms of $\eta$ and $C_G(g)$. However, it does not stay on the level of characters and conjugacy classes, but rather involves a sum over all elements $x$ of $G$ satisfying the somewhat awkward condition $x^m=(gx)^m$. This equation seems to defy nice structural analysis, it is just satisfied or not by $x\in G$, solutions do not form a subgroup, or conjugacy class, or other nice or nicely parametrized subset. Attempts to calculate indicators explicitly, and this includes calculating them with the help of computer algebra systems ---GAP \cite{GAP4} being a prime choice--- have to face the awkward equation and usually have to find, by nothing much more sophisticated than element-wise inspection, the solutions $x$ for given $g$.

The present paper will give an indicator formula that does not involve individual group elements, as it were, but rather only information ``on the character table level'', that is, the behavior of conjugacy classes and characters. There is a certain price to pay: We need to procure the character tables of all the centralizers of elements of $G$. But then, all the characters of all possible centralizers are needed to describe all the simples of $D(G)$ in the first place, so it should be no surprise if all this information is needed to calculate the indicators. To acknowledge the formula's drawback more clearly: To compute the indicators of the objects associated to the conjugacy class of only \emph{one} element $g$, we now need the character tables of \emph{all} (or at least some of) the other centralizers, not only of the centralizer of $g$.

Apart from being a new way to calculate indicators, the formula also leads to criteria testing whether all the indicators for the double of a given group are (rational) integers, or, in the terminology of Iovanov, Mason, and Montgomery \cite{MR3275646}, whether a given group is $FSZ$. These criteria do not involve all of the character tables of all centralizers, and are thus more efficient for their purpose than computing all the indicators outright.

We are able to demonstrate that the novel way to calculate indicators is vastly more efficient than previous methods when put to work in a computer algebra system such as GAP \cite{GAP4}. This relies, of course, on the fact that GAP is breathtakingly efficient in calculating character tables, and thus on the hard work put at the present author's fingertips by the creators of GAP. When comparing our method to the previously available methods, this reliance on GAP is not an issue, simply because knowledge of the character tables of centralizers is necessary anyhow if one wants to describe the simple modules of the Drinfeld double. 

The present paper contains no precise analysis of the efficiency of our GAP code in comparison to the code developed, say, in Courter's thesis \cite{MR3103664} or used for example in \cite{MR3275646}. We believe, however, that the few comparisons that we will draw in \cref{sec:applications} are convincing enough without precise analysis: We can deal with groups so much larger than those previously manageable, and the time needed for treating examples that were previously considered very large is such a small fraction of the time needed using the previously available methods that it seems out of the question to attribute all but an insignificant part of the improvement to the usual development of hardware over time. Once we consider the operations of finding conjugacy classes and character tables (and, as we shall find necessary, conjugating concrete elements into each other) as ``provided'', it is clear where the immense efficiency leap comes from: In a reasonably noncommutative group, there can be many orders of magnitude between the number of conjugacy classes (also of the centralizers in that group) and the number of elements of the group.

We will thus put an algorithm based on the indicator formula into GAP code. A certain amount of data management seems useful to avoid having to calculate certain ingredients needed for the calculation over and over again for the different indicators.

Motivated by a result on representations of symmetric groups by Scharf \cite{MR1113784}, Rebecca Courter \cite{MR3103664} investigated the conjecture that the indicators of simple modules over the double of a symmetric group (proven to be integers in \cite{MR3275646}) are nonnegative. She verified this experimentally up to $S_{10}$ and had to give up the pursuit on hardware available to her beyond that. With our approach can verify the conjecture up to $S_{18}$. 

We also verified ``naively'' (by computing all the indicators) that every simple group of order less than $200,000,000$ is $FSZ$. Here the author's laptop refused to continue, stalling during the computation of a character table. The question as to whether every simple group is an $FSZ$-group was put forward as ``[t]he most tantalizing open question'' on $FSZ$-groups in \cite{MR3275646}. Since our rationality criteria do not need \emph{all} the character tables (and can do with storing much less information at a time), they allow us to pursue the question further than the order cited above, and indeed we find that the exceptional Chevalley group $G_2(5)$ of order $5,859,000,000$ is not $FSZ_5$. It is the only counterexample we found while continuing the search up to order
$499,631,102,880$.

{\bf Acknowledgements:} The author thanks Susan Montgomery and Marc Keilberg for their comments, and Rebecca Courter for sharing her GAP code and data on the doubles of symmetric groups. Some calculations were performed using HPC resources from DSI-CCUB (Université de Bourgogne).
\section{Preliminaries and notation}
\label{sec:prel-notat}

We write $t\adhit x:=txt\inv $ for conjugation in a group, but often $x^G$ for a conjugacy class. $\Cl(G)$ denotes the set of conjugacy classes of $G$, and $\Irr(G)$ the set of irreducible complex characters. If $H\leq G$ is a subgroup and $t\in G$, so $t\adhit H$ is a conjugate subgroup, then for a character $\eta$ of $H$ we denote by $t\adhit\eta$ the character of $t\adhit H$ defined by $(t\adhit \eta)(x)=\eta(t\inv\adhit x)$. 

We denote by $\adams r$ the Adams operator on the algebra of class functions on $G$, defined by $\adams r(\chi)(g)=\chi(g^r)$. When $r$ is coprime to the exponent of $G$, then $\adams r$ preserves characters, and $\adams r(\chi)(g)=\sigma_r(\chi(g))$ where $\sigma_r$ is the automorphism of the cyclotomic field $\CF{\exp(G)}$ defined by sending a root of unity to its $r$-th power. We write $o(g)$ for the order of a group element $g$.
\section{Counting formulas}
\label{sec:counting-formulas}

As we shall review below, a key ingredient in the formula from \cite{MR2213320} computing Frobenius-Schur indicators of the modules over the Drinfeld double $D(G)$ of a finite group $G$ is the cardinality of the set
\begin{equation}
  \label{eq:1}
  G_m(g,z):=\{x\in G|x^m=(gx)^m=z\},
\end{equation}
where $g\in G$ is involved in the parametrization of the simple modules over $D(G)$. This section will deal with counting the elements of $G_m(g,z)$. We note that the equation $x^m=(gx)^m$ for fixed $g$ whose solutions determine the sets does not lend itself easily to group-theoretical reasoning, as it were: Its solutions do not form a subgroup or union of conjugacy classes. The condition is an equivalent rewriting of $g(x\adhit g)(x^2\adhit g)\dots(x^{m-1}\adhit g)=e$, which is the form used in \cite{MR2213320}; but the mapping, depending on $x$, whose ``kernel'' should thus contain $g$ is not a group homomorphism and lacks good properties (although expressions of this type are not unknown in group theory, see \cite{MR0108532}). 

We will derive counting formulas for $G_m(g,z)$ that are character theoretic: They are not based on inspecting all the elements of $G$ for picking out solutions, but on manipulations on the level of conjugacy classes, characters, and class functions ---on the centralizers of elements of $G$.

We start by noting, see \cite{MR3275646}, that $G_m(g,z)$ is empty, unless $g$ and $z$ commute in $G$, or in other words, belong to each other's centralizer in $G$.
As it turns out, the first step to successfully dealing with our problem is very simple: Instead of looking at $|G_m(g,z)|$ as functions of $z$, parametrized by $g$, we consider these cardinalities as functions of $g$ parametrized by $z$. To emphasize this trivial change of viewpoint notationally, we define
\begin{equation}
  \label{eq:2}
  \gamma_m^z\colon C_G(z)\to \CC;g\mapsto |G_m(g,z)|,
\end{equation}
and we observe:
\begin{LemDef}
  The map $\gamma_m^z$ is a class function on $C_G(z)$. Thus we can write
  \begin{equation}\label{eq:3}
    \gamma_m^z=\sum_{\chi\in\Irr(C_G(z))}\be m{z,\chi}\chi
  \end{equation}
  for $g\in C_G(z)$, with coefficients $\be m{z,\chi}\in\CC$.
\end{LemDef}

Further, for a group $H$ and $a,b,c\in H$, let $\CMC Habc$ denote the class multiplication coefficient which counts how many ways there exist to write $c\in H$ as a product $xy$ where $x$ and $y$ belong to the conjugacy classes $a^H$ and $b^H$, respectively. This depends only on the conjugacy class of $c$ (and by definition only on the conjugacy classes of $a$ and $b$.)

\begin{Lem}\label{prop:1}
  Assume that $g$ and $z$ commute in $G$, and $m$ is an integer. Then
  \begin{equation}
    \label{eq:4}
    \gamma_m^z(g)=\sum_{\substack{a,b\in\rCl(C_G(z))\\a^m=b^m=z}}\CMC{C_G(z)}a{b\inv}g.
  \end{equation}
  for any set $\rCl(C_G(z))$ of representatives of the conjugacy classes of $C_G(z)$.
\end{Lem}
\begin{proof}
    Whenever $x^m=(gx)^m=z$, then clearly $x$ and $gx$ commute with $z$, so
  \begin{align*}
    |\{x\in G&x^m=(gx)^m=z\}|\\
    &=\left|\left\{\left.(x,y)\in C_G(z)^2\right|x^m=y^m=z\;\land\; y=gx\right\}\right|\\
    &=\sum_{\substack{a,b\in\rCl(C_G(z))\\a^m=b^m=z}}\left|\left\{\left.(x,y)\in b^{C_G(z)}\times a^{C_G(z)}\right|g=yx\inv\right\}\right|\\
    &=\sum_{\substack{a,b\in\rCl(C_G(z))\\a^m=b^m=z}}\CMC{C_G(z)}a{b\inv}g,
  \end{align*}
\end{proof}

The formula of the lemma has the advantage of expressing the cardinalities of the $G_m$ sets in terms of information on the character table level: Not individual elements, but rather their conjugacy classes are summed over. The class multiplication coefficients are also defined on the conjugacy class level. In fact the class multiplication coefficients contain essentially the same information as the character table (the two data determine each other, this is discussed in detail in \cite[Ch.21]{MR1183468}).

We will proceed to use Burnside's classical formula expressing class multiplication coefficients in terms of character values in order to derive a formula for $\gamma_m^z$ in terms of character values. In fact we will write out several easily equivalent such expressions for the coefficient $\be m{z,\chi}$ below. We could not quite decide which way of writing them we should prefer.

\begin{Def}
  Let $G$ be a finite group and $z\in G$. For an integer $m$, define 
  \begin{equation}
    w^z_m\colon C_G(z)\to\CC; x\mapsto \begin{cases}1 &\text{ if }x^m=z\\0&\text{ otherwise}
    \end{cases}
  \end{equation}
  It is easy to see that $w^z_m$ is a class function on $C_G(z)$.

  For a character $\chi$ of $C_G(z)$ and an integer $m$ let
  \begin{equation}\phi_m(\chi,z)=\sum_{\substack{x\in C_G(z)\\x^m=z}}\chi(x)=\sum_{\substack{a\in\rCl(C_G(z))\\a^m=z}}|a^{C_G(z)}|\chi(a).
  \end{equation}
\end{Def}
\begin{Lem}\label{prop:2}
  Let $G$ be a finite group, $z\in G$, and $\chi$ an irreducible character of $C_G(z)$. Then
\begin{align}
  \label{eq:5}
  \phi^z_m(\chi)&=|C_G(z)|\langle \chi,w^z_m\rangle\\
  \label{eq:6} |\phi^z_m(\chi)|&=|C_G(z)|\cdot|\chi w^z_m|\\
  \label{eq:7}\be m{z,\chi}&=\frac{|\phi^z_m(\chi)|^2}{|C_G(z)|\chi(e)}\\
  \label{eq:8}              &=\frac{|\langle\chi,w^z_m\rangle|^2|C_G(z)|}{\chi(e)}\\
  \label{eq:9}              &=\frac{|\chi w_m^z|^2|C_G(z)|}{\chi(e)}.
  \end{align}
\end{Lem}
\begin{proof}
  We use Burnside's formula \cite[Ch.21, Thm.1.2]{MR1183468} expressing class multiplication coefficients in terms of characters:
  \begin{equation*}
    \label{eq:10} \CMC{C_G(z)}a{b\inv}g=\frac{|a^{C_G(z)}|\cdot|b^{C_G(z)}|}{|C_G(z)|}\sum_{\chi\in\Irr(C_G(z))}\frac{\chi(a)\ol\chi(b)\ol\chi(g)}{\chi(e)}
  \end{equation*}
  giving
  \begin{align*} |G_m(g,z)|&=\sum_{\substack{a,b\in\rCl(C_G(z))\\a^m=b^m=z}}\frac{|a^{C_G(z)}|\cdot|b^{C_G(z)}|}{|C_G(z)|}\sum_{\chi\in\Irr(C_G(z))}\frac{\chi(a)\ol\chi(b)\ol\chi(g)}{\chi(e)}\\
                           &=\frac1{|C_G(z)|}\sum_{\substack{x,y\in C_G(z)\\x^m=y^m=z}}\sum_{\chi\in\Irr(C_G(z))}\frac{\chi(x)\ol{\chi(y)}\ol\chi(g)}{\chi(e)}\\
                           &=\frac1{|C_G(z)|}\sum_{\chi\in\Irr(C_G(z))}\frac{\ol\chi(g)}{\chi(e)}\sum_{\substack{x,y\in C_G(z)\\x^m=y^m=z}}\chi(x)\ol\chi(y)\\
                           &=\frac1{|C_G(z)|}\sum_{\chi\in\Irr(C_G(z))}\frac{\ol\chi(g)}{\chi(e)}\left(\sum_{\substack{x\in C_G(z)\\x^m=z}}\chi(x)\right)\left(\sum_{\substack{y\in C_G(z)\\y^m=z}}\ol\chi(y)\right)\\
                           &=\frac1{|C_G(z)|}\sum_{\chi\in\Irr(C_G(z))}\frac{\chi(g)}{\chi(e)}|\phi_m^z(\chi)|^2\\
  \end{align*}
This proves \eqref{eq:7}, and \eqref{eq:8}, \eqref{eq:9} are only variations based on \eqref{eq:6}, \eqref{eq:7}, which in their turn are trivial.
\end{proof}

\begin{Lem}\label{prop:3}
  Let $G$ be a finite group and $z\in G$. Put $e(z):=\exp(C_G(z))$. Then
  \begin{enumerate}
  \item $\gamma_{m+e(z)}^z=\gamma_m^z$ for all $m$.
  \item \label{item:1}If $(a,e(z))=1$ then $\gamma_{am}^{z^a}=\gamma_m^z$, $\adams a\gamma_m^{z^a}=\gamma_m^z$, and $\adams a\gamma_{m}^z=\gamma_{am}^z$.
  \item \label{item:3}If $m|e(z)$ and $mo(z)\not|e(z)$ then $\gamma_m^z=0$.
  \item \label{item:2}The class functions $\gamma_m^z$ are determined by those where $m|\frac{e(z)}{o(z)}$.
  \item \label{item:4}If $(m,e(z))=1$, then $\gamma_m^z(g)=\delta_{g,1}$.
  \end{enumerate}
\end{Lem}
\begin{proof}
  The periodicity in $m$ is trivial. If $(a,e(z))=1$, then $x^m=z$ is equivalent to $x^{am}=z^a$ and $(gx)^m=z$ equivalent to $(gx)^{am}=z^a$ for $x\in C_G(z)$, proving the first equality in \cref{item:1}. For the second, note that $w_m^z=\adams aw_m^{z^a}$ and calculate
  \begin{align*}
    \sum_{\chi\in\Irr(C_G(z))}\frac{|\langle\chi,w_m^z\rangle|^2}{\chi(e)}\chi
    &=\sum_{\chi\in\Irr(C_G(z))}\frac{|\langle\adams a\chi,w_m^z\rangle|^2}{\chi(e)}\adams a\chi\\
    &=\sum_{\chi\in\Irr(C_G(z))}\frac{|\adams a\chi,\adams a (w_m^{z^a})\rangle|^2}{\chi(e)}\adams a\chi\\
    &=\sum_{\chi\in\Irr(C_G(z))}\frac{|\langle\chi,w_m^{z^a}\rangle|^2}{\chi(e)}\adams a\chi\\
  \end{align*}
  so that $\adams a\gamma_m^{z^a}=\gamma_m^z$ and thus $\adams a\gamma_m^z=\adams a\gamma_{am}^{z^a}=\gamma_{am}^z$.

  If $m'=(m,e(z))$ then one can choose $a,b$ with $m'=am+be(z)$ such that $(a,e(z))=1$. (This is surely elementary and common knowledge, but for completeness: If $am\equiv m'\mod e(z)$ then also $(a+k\frac{e(z)}{m'})m\equiv m'\mod e(z)$ for every integer $k$. Choose $k$ such that every prime dividing $e(z)$ but not $\frac{e(z)}{m'}$ divides either $a$ or $k$ but not both; then $a+k\frac{e(z)}{m'}$ is relatively prime to $e(z)$.) Now $\gamma_{m'}^z=\gamma_{am}^z=\adams a\gamma_m^z$, so that $\gamma_m^z$ is determined by $\gamma_{m'}^z$.

  Assume $m|e(z)$ and $\gamma_m^z\neq 0$; thus there is $x\in C_G(z)$ with $x^m=z$. But then $z^{e(z)/m}=x^{e(z)}=1$ and thus $o(z)|(e(z)/m)$ proving \cref{item:3} and thus what remained of \cref{item:2}. Finally \cref{item:4} amounts to the trivial observation that $\gamma_1^z(g)=\delta_{g,1}$
\end{proof}

\section{Indicator formulas for characters of the double}
\label{sec:indic-form-char}

We recall that simple modules over the Drinfeld double $D(G)$ of a finite group are parametrized by pairs $(g,\eta)$ where $g\in G$ and $\eta\in\Irr(C_G(g))$. Conjugate elements $g$ yield isomorphic simple modules: If $t,g\in G$ then $C_G(t\adhit g)=t\adhit C_G(g)$, and thus (irreducible) characters of $C_G(t\adhit g)$ are in bijection with (irreducible) characters of $C_G(g)$; the (simple) $D(G)$-module corresponding to $(g,\eta)$ is isomorphic to the one corresponding to $(t\adhit g,t\adhit\eta)$.

The following lemma sums up the formulas for the indicators of simple modules for the Drinfeld double of a finite group as found in \cite{MR2213320,MR3275646}:
\begin{Lem}
  Let $G$ be a finite group, $g\in G$, and $\eta$ a character of the centralizer $C_G(g)$. Then the $m$-th Frobenius-Schur indicator of the $D(G)$-module associated to $g$ and $\eta$ is
  \begin{align}
    \label{eq:11}
    \nu_m(g,\eta)&=\frac1{|C_G(g)|}\sum_{\substack{x\in G\\x^m=(gx)^m}}\eta(x^m)\\
                 &=\frac1{|C_G(g)|}\sum_{z\in C_G(g)}|G_m(g,z)|\eta(z)\label{eq:12}.
  \end{align}
\end{Lem}

After the results of the previous section, we already have a formula for the indicators of the doubles of finite groups that is based solely on conjugacy classes and characters:
\begin{Cor}\label{cor1}
  Let $G$ be a finite group, $g\in G$, and $\eta$ a character of $C_G(g)$. Then
  \begin{align}\label{eq:13} \nu_m(g,\eta)
    &=\frac1{|C_G(g)|}\sum_{z\in C_G(g)}\gamma_m^z(g)\eta(z)\\
    &=\frac1{|C_G(g)|}\sum_{z\in\rCl(C_G(g))}|z^{C_G(g)}|\gamma_m^z(g)\eta(z)\\
    \intertext{ with }
 \gamma_m^z&=|C_G(z)|\sum_{\chi\in\Irr(C_G(z))}\frac{|\langle\chi,w^z_m\rangle|^2}{\chi(e)}\chi.
  \end{align}
  where $w_m^z(x)=\delta_{x^m,z}$.
\end{Cor}

We note that to calculate the indicators of the objects associated to one conjugacy class (of g), we need to have the character tables of other elements as well. If we are interested in the indicators for one specific object (or conjugacy class), this may be a serious drawback in comparison to the formula \eqref{eq:11}. If, however, we are interested in all the indicators of the modules for the double of $G$, then we need ``all'' centralizers and their characters just to describe the objects. After this attempt to defend the merits of \cref{cor1}, we need to point out an aspect that needs to be improved: To describe the simple modules of the double, we need the character table of the centralizer of one representative of each conjugacy class of $G$. But \eqref{eq:13} involves the character table of the centralizer of a representative of each conjugacy class in $C_G(g)$; distinct such representatives may be conjugate in $G$. Their centralizers are thus conjugate in $G$, but not equal. The formula requires the character tables of all these conjugate centralizers separately, thus, if put blindly into GAP code, say, requires the computation of many redundant character tables.

We will now rewrite the indicator formula \eqref{eq:13} in terms of characters of the Drinfeld double of $G$, after recalling briefly the character theory of $D(G)$.

Let $\cp(G):=\{(x,y)\in G^2|xy=yx\}$ denote the set of pairs of
commuting elements in $G$. The group $G$ acts on $\cp(G)$ by
simultaneous conjugation $t\adhit (x,y)=(t\adhit x,t\adhit y)$ and we denote
by $\cpc(G)$ the set of orbits under this action; we will write $(x,y)^G$ for the orbit of $(x,y)$. By $\rcpc(G)$ we
will denote a chosen set of representatives of $\cpc(G)$.

A \dcf\ is a function $f\colon\cp(G)\to\CC$ invariant under
simultaneous conjugation, i.~e.\ a function that is invariant on the classes in
$\cpc(G)$.

A representation of the Drinfeld double $D(G)$ has a character which
is a \dcf. If the representation is described by the pair $(g,\eta)$
with $\eta$ a character of $C_G(g)$, then the corresponding \dcf\ is
\begin{align}
  \Xi\colon G\times G&\to\CC\notag\\
  (x,y)&\mapsto\label{eq:14}
         \begin{cases}
           \chi(tyt\inv)&\text{ if }txt\inv=y\\
           0&\text{ if }x\text{ is not conjugate to }y.
         \end{cases}
\end{align}
The characters of the irreducible representations of $D(G)$ form a basis of the space of \dcfs.

\begin{Lem}
  Let $\Xi$ be a character of the Drinfeld double $D(G)$ of a finite
  group $G$, and $m$ an integer. Then
  \begin{align}\label{eq:16} \nu_m(\Xi)&=\sum_{(h,z)\in\rP(G)}\frac1{|C_G(h,z)|}
                            \gamma^z_m(h)
                                         \Xi(h,z)\\
    &=\frac1{|G|}\sum_{(h,z)\in P(G)}\notag
                            \gamma^z_m(h)
                                         \Xi(h,z)
  \end{align}
\end{Lem}
\begin{proof}
  It suffices to show the claim when $\Xi$ is the character associated to $g\in G$ and $\eta\in\Irr(C_G(g))$.

  We have $\gamma_m^{t\adhit z}=t\adhit\gamma_m^z$ for $z,t\in G$ since $G_m(t\adhit g,t\adhit z)=t\adhit G_m(g,z)$ for $(g,z)\in\cp(G)$.
  Therefore,
  \begin{equation*} \sum_{(h,z)\in\rcpc(G)}\frac1{|C_G(h,z)|}
    \gamma_m^z(h)
    \Xi(h,z)
    =
    \frac1{|G|}\sum_{(h,z)\in\cp(G)}
    \gamma_m^z(h)
    \Xi(h,z)
  \end{equation*}
  is independent of the choice $\rcpc(G)$ of representatives, and it suffices to prove the first formula for one particular choice of representatives; we opt for
  \begin{equation*}
    \rcpc(G)=\bigsqcup_{f\in\rCl(G)}\{f\}\times\rCl(C_G(f)).
  \end{equation*}
  where $\rCl(G)$ is chosen to contain $g$. But then
  \begin{multline*} \sum_{(h,z)\in\rcpc(G)}\frac1{|C_G(h,z)|}
    \gamma_m^z(h)
    \Xi(h,z)
    \\=\sum_{h\in\rCl(G)}\sum_{z\in\rCl(C_G(h))}\frac1{|C_G(h,z)|}
    \gamma_m^z(h)
    \Xi(h,z)
    \\=\sum_{z\in\rCl(C_G(g))}\frac1{|C_G(g,z)|}
    \gamma_m^z(g)\eta(z)=\nu_m(g,\eta)=\nu_m(\Xi)
  \end{multline*}
  since $|z^{C_G(g)}|\cdot|C_G(g,z)|=|z^{C_G(g)}|\cdot|C_{C_G(g)}(z)|=|C_G(g)|$ and $\Xi(h,z)=\delta_{g,h}\eta(z)$ for our choice of $\rcpc(G)$.
\end{proof}

\begin{Rem}
  As pointed out in \cite{MR3275646}, we can reinterpet \eqref{eq:12} as
  \begin{equation*}
    |G_m(g,z)|=\sum_{\eta\in\Irr(C_G(g))}\nu_m(g,\eta)\eta(z)
  \end{equation*}
  which we can rewrite in turn, by adding vanishing terms, as
  \begin{equation*}
    |G_m(g,z)|=\sum_{\Xi\in\Irr(D(G))}\nu_m(\Xi)\Xi(g,z).
  \end{equation*}
  Thus, the Frobenius-Schur indicators are the coefficients needed to develop the \dcf\ $|G_m(g,z)|$ in terms of irreducible characters. In the same way, we can extend~\eqref{eq:3} to read
  \begin{equation*}
    |G_m(g,z)|=\sum_{\Xi\in\Irr(D(G))}\be m\Xi\Xi(z,g)
  \end{equation*}
  by parametrizing the irreducible characters of $D(G)$ by pairs $(z',\chi)$ with $z'\in\rCl(G)$ and $\rCl(G)$ chosen to contain $z$, and $\chi\in\Irr(C_G(z'))$.

  Thus, both the indicators $\nu_m(\Xi)$ and the somewhat more easily calculable $\be m\Xi$ result from writing the \dcf\ $|G_m|$ as a linear combination of irreducible characters. The only difference is the order of the two arguments of $G_m$. Now switching the two arguments of an irreducible character of $D(G)$ is a way to describe the S-matrix that is part of the modular data of the modular category of irreducible representations of the Drinfeld double. Thus, the vector of the values $\nu_m(\Xi)$ of the $m$-th indicator function on all the irreducible characters of $D(G)$ is related to the vector of the values $\be m\Xi$ by the action of the S-matrix on the irreducible characters. The indicator formula \eqref{eq:16} is the result of making this transformation between $\nu_m$ and $\beta_m$ explicit. Note that \eqref{eq:16} is computationally less costly than it would be to obtain all the coefficients of the S-matrix. This relates to the fact that the vectors of all the $m$-th indicators seem to be a weaker invariant than the modular data: While the S-matrix is invertible, many irreducible characters of $D(H)$ usually share the same indicator values (for all $m$).
\end{Rem}

\begin{Rem}
  For any modular fusion category, there is a way of calculating the Frobenius-Schur indicators in terms of the modular data: The Verlinde formula (cf.~\cite{BakKir:LTCMF}) expresses the fusion coefficients in terms of the S-matrix:
    \begin{equation}
    \label{eq:10}
    N_{ik}^j= \frac1{\dim(\mathcal A)}\sum_{r} \frac{S_{ir}S_{kr}S_{\bar{j} r}}{S_{0
r}},
\end{equation}
(where $S_{ij}$ is the trace of the double braiding on $X_i^*\ot X_j$, the tensor product of two simples, and $N_{ik}^j$ is the multiplicity of $X_j$ in $X_i\ot X_k.)$
The Bantay-type formula for indicators from \cite{NgSch:FSIESC} expresses the indicators in terms of the fusion coefficients and the S-\ and T-matrices, in our special case:
\begin{equation}
  \label{eq:9}
    \nu_m(X_j)= \frac{1}{|G|^2}\sum_{i, k}N_{ik}^j d_i d_k  \left(\frac{\omega_i}{\omega_k}\right)^m,
  \end{equation}
  where $\omega_i$ are the components of the (diagonal) T-matrix, or the ribbon structure, and $d_j$ is the dimension of $X_j$.
Using this approach directly for doubles of finite groups is probably more involved than using our formulas, since the calculation of the $S$-matrix is ``expensive''; however, one may suspect that this way of calculating the indicators might lead to our formulas after suitable simplifications. A structural similarity is certainly that the Bantay-type formula involves all the simple objects of a modular fusion category in the calculation of the indicators of one of them. Although it would be interesting to derive our formula this way (despite the fact that our proof needs no more than the formula from \cite{MR2213320} and some classical character theory of finite groups), we did not pursue the question. In this context it is perhaps worth noting that several authors \cite{MR980656,MR2164398} have pointed out the similarity between Burnside's formula used in the proof of \cref{prop:1} (albeit for the centralizers, not only the group itself) and the Verlinde formula.
\end{Rem}

\section{Computing indicators through character tables}
\label{sec:comp-indic}

If we want to explicitly calculate indicators, then it is convenient to use a somewhat ``opposite'' parametrization of orbits of commuting pairs to the one used before, namely
\begin{equation}
  \rcpc(G)=\bigsqcup_{z\in\rCl(G)}\rCl(C_G(z))\times\{z\}
\end{equation}
which yields
\begin{align}
  \label{eq:17}              \nu_m(\Xi)&=\sum_{z\in\rCl(G)}\sum_{h\in\rCl(C_G(z))}\frac1{|C_G(h,z)|}\sum_{\chi\in\Irr(C_G(z))}\be m{z,\chi}\chi(h)\Xi(h,z)
\end{align}
This formula is somewhat deceptively simple: When we specialize it to the character $\Xi$ afforded by $g\in G$ and $\eta\in\Irr(C_G(g))$, we need to go back to \eqref{eq:14} in calculating $\Xi(h,z)$ by checking if, and how, $g$ and $h$ are conjugate:
\begin{Prop}
  Let $G$ be a finite group. Choose a set $\rCl(G)$ of representatives of the conjugacy classes of $G$, and, for each $z\in\rCl(G)$, a set $\rCl(C_G(z))$ of representatives for the conjugacy classes of $C_G(z)$. For $h\in\rCl(C_G(z))$ choose an element $t(h,z)$ such that $g(h):=t(h,z)\adhit h\in\rCl(G)$.
Then
  \begin{align} \nu_m(\Xi)&=\sum_{z\in\rCl(G)}\sum_{h\in\rCl(C_G(z))}\frac1{|C_G(h,z)|}\gamma_m^z(h)\Xi(g(h),t(h,z)\adhit z)\\
    \nu_m(g,\eta)&=\frac1{|C_G(g)|}\sum_{z\in\rCl(G)}\sum_{\substack{h\in\rCl(C_G(z))\\g(h)=g}}|z^{C_G(h)}|\gamma_m^z(h)\eta(t(h,z)\adhit z).\label{eq:18}
  \end{align}
for a character $\Xi$ of $D(G)$, or for $g\in G$ and $\eta\in\Irr(C_G(g))$, respectively.
\end{Prop}

The last two formulas are ``convenient'' for practical calculations in the following sense: One needs to know only as many explicit character tables as are needed to describe the simple modules of the double, namely, the character table of $C_G(g)$ for $g$ in a cross section of the conjugacy classes of $G$. Then, one has to determine systems of representatives for the conjugacy classes of each centralizer, and one has to know how the representatives of the conjugacy classes of $C_G(g)$ are conjugate to the chosen representatives of the classes of $G$ (in particular, how the conjugacy classes of the centralizers fall into the conjugacy classes of $G$). We have already explained that $\be m{z,\chi}$ is computed in terms of class functions and conjugacy classes of centralizers, whence so is the class function $\gamma_m^z$; all that has to be done beyond is applying ordinary group characters or class functions to group elements.

We note one final variant of the indicator formula; if we linearly extend group characters to the group algebra, then we can write the indicator of the object associated to the pair $(g,\eta)$ as the value of $\eta$ on a suitable element of $\CC[C_G(g)]$. The formula for the element $\mu_m(g)$ below would easily admit being ``patched'' together for all values of $g$ to obtain an element $\mu_m$ of the double $D(G)$ such that $\nu_m(\Xi)=\Xi(\mu_m)$ for any character $\Xi$ of the double. The only reason we did not write this down is that this would be the only point in the paper where we need elements of the Drinfeld double, and would thus have to fix the necessary conventions. Note also that the unique central element of the Drinfeld double mapped to the $m$-th indicator by the character of a representation is the $m$-th Sweedler power of the integral \cite{MR1808131}. Patching the elements $\mu_m(g)$ below would not yield that central element because we chose representatives rather than sums over class elements.
\begin{Cor}\label{cor:mu}
  For $g\in G$ define $\mu_m(g)\in\CC[C_G(g)]$ by
  \begin{equation} \label{eq:19}
    \mu_m(g):=
\sum_{z\in\rCl(G)}\sum_{\substack{h\in\rCl(C_G(z))\\g(h)=g}}|z^{C_G(h)}|\gamma_m^z(h)(t(h,z)\adhit z)\in\CC[C_G(g)].
\end{equation}
   Then $\nu_m(g,\eta)=\eta(\mu_m(g))$ for $\eta\in\Irr(C_G(g))$.
\end{Cor}

\section{Implementation}
\label{sec:implementation}

We have implemented the indicator formulas obtained in the previous section in GAP \cite{GAP4}. Since oddly enough GAP does not provide the absolute value of cyclotomic numbers, we have to provide it, along with a ``Kronecker delta'' function; as a further piece of code less specific to our problems we provide a class function whose value, for a character table and an element of the underlying group, is the number in the table of the conjugacy class containing the element.
\begin{lstlisting}[linerange=triv-end]
#& triv #&

SquaredModulus:=function(z)
  return z*ComplexConjugate(z);
end;

krondel:=function(x,y)
  if x=y then return 1;
  else return 0;
  fi;
end;

classpos:=function(CT)
  return 
    ClassFunction(
       CT,
       [1..Size(ConjugacyClasses(CT))]);
end;

#& end #&

#& wbetagamma #&

cfw:=function(CT,z,m)
  return 
    ClassFunction(CT,List(ConjugacyClasses(CT),
       c->krondel(Representative(c)^m,z)));
end;

beta:=function(CT,z,m,chi)
  local skp;
  skp:=ScalarProduct(CT,chi,cfw(CT,z,m));
  return SquaredModulus(skp)
         *Size(UnderlyingGroup(CT))
         /DegreeOfCharacter(chi);
end;

gamma:=function(CT,z,m)
  local chi,ichi,clf;
  clf:=List([1..Size(ConjugacyClasses(CT))],i->0);
  for ichi in [1..Size(ConjugacyClasses(CT))] do
    chi:=Irr(CT)[ichi];
    clf:=clf+beta(CT,z,m,chi)*chi;
  od;
  return clf;
end;

#& end #&

#& makedfsi #&

MakeDFSI:=function(GT)
  if IsBound(GT!.CTlist) then return;fi;
  GT!.exp:=Exponent(UnderlyingGroup(GT));
  GT!.div:=DivisorsInt(GT!.exp);
  GT!.CTlist:=[];
  GT!.classposlist:=[];
  GT!.gammalist:=List(ConjugacyClasses(GT),c->[]);
  GT!.mulist:=List(GT!.div,m->[]);
  GT!.mates:=List(ConjugacyClasses(GT),c->[]);
end;

#& end #&

#& providecharactertable #&

providecharactertable:=function(GT,ig)
  if IsBound(GT!.CTlist[ig]) then return;
  fi;
  GT!.CTlist[ig]
    :=OrdinaryCharacterTable(
         Centralizer(
            UnderlyingGroup(GT),
            Representative(
                    ConjugacyClasses(GT)[ig])));
  Irr(GT!.CTlist[ig]);;
  GT!.classposlist[ig]
    :=classpos(GT!.CTlist[ig]);
  return;
end;

#& end #&

#& providegamma #&

providegamma:=function(GT,iz,im)
  local x;
  if IsBound(GT!.gammalist[iz][im]) then return;
  fi;
  GT!.gammalist[iz][im]:=
    gamma(GT!.CTlist[iz],
       Representative(ConjugacyClasses(GT)[iz]),
          GT!.div[im]);
  return;
end;

#& end #&

#& providemate #&

providemate:=function(GT,iz,ih)
    local ig,t;
    if IsBound(GT!.mates[iz][ih]) then return;
    fi; 
    ig:=FusionConjugacyClasses(GT!.CTlist[iz],GT)[ih];
    t:=RepresentativeAction(
            UnderlyingGroup(GT),
            Representative(
               ConjugacyClasses(GT!.CTlist[iz])[ih]),
            Representative(
               ConjugacyClasses(GT)[ig]));
      
    GT!.mates[iz][ih]:=
      (Representative(ConjugacyClasses(GT)[iz])^t)
      ^GT!.classposlist[ig];
end;

#& end #&

#& providemu #&

providemu:=function(GT,ig,im)
  local nmulist,iz,izc,ih,izset,ez,oz,merep,mm,imm,a,c,ihc;
  if IsBound(GT!.mulist[im][ig]) then
    return;
  fi;
  providecharactertable(GT,ig);
  nmulist:=List(ConjugacyClasses(GT!.CTlist[ig]),i->0);
  izset:=Set([1..Size(ConjugacyClasses(GT!.CTlist[ig]))],
             iz->FusionConjugacyClasses(GT!.CTlist[ig],GT)[iz]);
  for iz in izset do
      providecharactertable(GT,iz);
      ez:=Exponent(UnderlyingGroup(GT!.CTlist[iz]));
      oz:=Order(Representative(ConjugacyClasses(GT)[iz]));
      merep:=Gcdex(GT!.div[im],ez);
      mm:=merep.gcd;
      imm:=Position(GT!.div,mm);
      if not RemInt(ez,mm*oz)=0 then continue;
      fi;
      a:=merep.coeff1;
      cc:=Gcdex(a,ez);
      while cc.gcd<>1 do
          a:=a+merep.coeff3;
          cc:=Gcdex(a,ez);
      od;
      c:=cc.coeff1;
      providegamma(GT,iz,imm);
      for ih in [1..Size(ConjugacyClasses(GT!.CTlist[iz]))] do
      if FusionConjugacyClasses(GT!.CTlist[iz],GT)[ih]<>ig 
         then continue;
      fi;
      providemate(GT,iz,ih);
      izc:=GT!.mates[iz][ih];
      ihc:=PowerMap(GT!.CTlist[iz],c,ih);
      nmulist[izc]:=nmulist[izc]+
                    SizesConjugacyClasses(GT!.CTlist[ig])[izc]
                    *GT!.gammalist[iz][imm][ihc];
    od;
  od;
  GT!.mulist[im][ig]:=nmulist;
end;

#& end #&

#& doubleindicator #&

DoubleIndicator:=function(GT,ig,ieta,m)
  local im;    
  MakeDFSI(GT);
  im:=Position(GT!.div,m);
  if ig=1 then 
      return Indicator(GT,m)[ieta];
  fi;
  providemu(GT,ig,im);
  return 1/Size(UnderlyingGroup(GT!.CTlist[ig]))*
         Sum([1..Size(ConjugacyClasses(GT!.CTlist[ig]))],
            iz->Irr(GT!.CTlist[ig])[ieta][iz]
            *GT!.mulist[im][ig][iz]);
end;

#& end #&

\end{lstlisting}

Next, we program the class function $w_m^z$ as a GAP function \lstinline!cfw!, the coefficient $\beta_m(z,\chi)$ as a GAP function \lstinline!beta!, and the class function $\gamma_m^z$ as a GAP function \lstinline!gamma!. All three depend on a character table \lstinline!CT! (the character table of the centralizer $C_G(z)$) and the element $z$; \lstinline!beta! also depends on the character \lstinline!chi!.
\begin{lstlisting}[linerange=wbetagamma-end]
#& triv #&

SquaredModulus:=function(z)
  return z*ComplexConjugate(z);
end;

krondel:=function(x,y)
  if x=y then return 1;
  else return 0;
  fi;
end;

classpos:=function(CT)
  return 
    ClassFunction(
       CT,
       [1..Size(ConjugacyClasses(CT))]);
end;

#& end #&

#& wbetagamma #&

cfw:=function(CT,z,m)
  return 
    ClassFunction(CT,List(ConjugacyClasses(CT),
       c->krondel(Representative(c)^m,z)));
end;

beta:=function(CT,z,m,chi)
  local skp;
  skp:=ScalarProduct(CT,chi,cfw(CT,z,m));
  return SquaredModulus(skp)
         *Size(UnderlyingGroup(CT))
         /DegreeOfCharacter(chi);
end;

gamma:=function(CT,z,m)
  local chi,ichi,clf;
  clf:=List([1..Size(ConjugacyClasses(CT))],i->0);
  for ichi in [1..Size(ConjugacyClasses(CT))] do
    chi:=Irr(CT)[ichi];
    clf:=clf+beta(CT,z,m,chi)*chi;
  od;
  return clf;
end;

#& end #&

#& makedfsi #&

MakeDFSI:=function(GT)
  if IsBound(GT!.CTlist) then return;fi;
  GT!.exp:=Exponent(UnderlyingGroup(GT));
  GT!.div:=DivisorsInt(GT!.exp);
  GT!.CTlist:=[];
  GT!.classposlist:=[];
  GT!.gammalist:=List(ConjugacyClasses(GT),c->[]);
  GT!.mulist:=List(GT!.div,m->[]);
  GT!.mates:=List(ConjugacyClasses(GT),c->[]);
end;

#& end #&

#& providecharactertable #&

providecharactertable:=function(GT,ig)
  if IsBound(GT!.CTlist[ig]) then return;
  fi;
  GT!.CTlist[ig]
    :=OrdinaryCharacterTable(
         Centralizer(
            UnderlyingGroup(GT),
            Representative(
                    ConjugacyClasses(GT)[ig])));
  Irr(GT!.CTlist[ig]);;
  GT!.classposlist[ig]
    :=classpos(GT!.CTlist[ig]);
  return;
end;

#& end #&

#& providegamma #&

providegamma:=function(GT,iz,im)
  local x;
  if IsBound(GT!.gammalist[iz][im]) then return;
  fi;
  GT!.gammalist[iz][im]:=
    gamma(GT!.CTlist[iz],
       Representative(ConjugacyClasses(GT)[iz]),
          GT!.div[im]);
  return;
end;

#& end #&

#& providemate #&

providemate:=function(GT,iz,ih)
    local ig,t;
    if IsBound(GT!.mates[iz][ih]) then return;
    fi; 
    ig:=FusionConjugacyClasses(GT!.CTlist[iz],GT)[ih];
    t:=RepresentativeAction(
            UnderlyingGroup(GT),
            Representative(
               ConjugacyClasses(GT!.CTlist[iz])[ih]),
            Representative(
               ConjugacyClasses(GT)[ig]));
      
    GT!.mates[iz][ih]:=
      (Representative(ConjugacyClasses(GT)[iz])^t)
      ^GT!.classposlist[ig];
end;

#& end #&

#& providemu #&

providemu:=function(GT,ig,im)
  local nmulist,iz,izc,ih,izset,ez,oz,merep,mm,imm,a,c,ihc;
  if IsBound(GT!.mulist[im][ig]) then
    return;
  fi;
  providecharactertable(GT,ig);
  nmulist:=List(ConjugacyClasses(GT!.CTlist[ig]),i->0);
  izset:=Set([1..Size(ConjugacyClasses(GT!.CTlist[ig]))],
             iz->FusionConjugacyClasses(GT!.CTlist[ig],GT)[iz]);
  for iz in izset do
      providecharactertable(GT,iz);
      ez:=Exponent(UnderlyingGroup(GT!.CTlist[iz]));
      oz:=Order(Representative(ConjugacyClasses(GT)[iz]));
      merep:=Gcdex(GT!.div[im],ez);
      mm:=merep.gcd;
      imm:=Position(GT!.div,mm);
      if not RemInt(ez,mm*oz)=0 then continue;
      fi;
      a:=merep.coeff1;
      cc:=Gcdex(a,ez);
      while cc.gcd<>1 do
          a:=a+merep.coeff3;
          cc:=Gcdex(a,ez);
      od;
      c:=cc.coeff1;
      providegamma(GT,iz,imm);
      for ih in [1..Size(ConjugacyClasses(GT!.CTlist[iz]))] do
      if FusionConjugacyClasses(GT!.CTlist[iz],GT)[ih]<>ig 
         then continue;
      fi;
      providemate(GT,iz,ih);
      izc:=GT!.mates[iz][ih];
      ihc:=PowerMap(GT!.CTlist[iz],c,ih);
      nmulist[izc]:=nmulist[izc]+
                    SizesConjugacyClasses(GT!.CTlist[ig])[izc]
                    *GT!.gammalist[iz][imm][ihc];
    od;
  od;
  GT!.mulist[im][ig]:=nmulist;
end;

#& end #&

#& doubleindicator #&

DoubleIndicator:=function(GT,ig,ieta,m)
  local im;    
  MakeDFSI(GT);
  im:=Position(GT!.div,m);
  if ig=1 then 
      return Indicator(GT,m)[ieta];
  fi;
  providemu(GT,ig,im);
  return 1/Size(UnderlyingGroup(GT!.CTlist[ig]))*
         Sum([1..Size(ConjugacyClasses(GT!.CTlist[ig]))],
            iz->Irr(GT!.CTlist[ig])[ieta][iz]
            *GT!.mulist[im][ig][iz]);
end;

#& end #&

\end{lstlisting}
Note that our version of \lstinline!gamma! does not take advantage of the redundancies observed in \cref{prop:3}. We will use these before calling the function.

Before writing the actual functions calculating indicators, let us note that a large amount of auxiliary information is needed to do so: mainly centralizers and their characters, the class functions $\gamma_m^x$, and the result of conjugating elements into each other's centralizers. To calculate one indicator, not all of this information may be required\footnote{An unnecessary caveat, of course, if we calculate \emph{all} the indicators for the double of a certain group; however, computers caught between an inept administrator and an undisciplined user, be they the same person, do sometimes break down, and thus calculations for large groups may end up being done in several portions.}, but it would be a good idea to keep any that have already been obtained while calculating one indicator, lest they be needed again for the calculation of another. Thus, whenever we start calculating indicators for the double of the underlying group of a character table, we will enrich this character table by adding several (record) fields of useful data for indicator calculations, which do not at first contain a lot of information, but can be used for storing it. More precisely, we will compute indicators for the double of the underlying group $G$ of a character table \lstinline!GT!. We will provide for storing with \lstinline!GT! the following information: the exponent of $G$ and its divisors; a list \lstinline!CTlist! of the character tables of the centralizers in $G$; a list of the class functions that compute the list position of the conjugacy class of an element in a centralizer; a list of lists of the class functions $\gamma_m^z$; a list of lists of the elements $\mu_m$ from \cref{cor:mu} and a list of lists of ``mates'' that we will explain below.
\begin{lstlisting}[linerange=makedfsi-end]
#& triv #&

SquaredModulus:=function(z)
  return z*ComplexConjugate(z);
end;

krondel:=function(x,y)
  if x=y then return 1;
  else return 0;
  fi;
end;

classpos:=function(CT)
  return 
    ClassFunction(
       CT,
       [1..Size(ConjugacyClasses(CT))]);
end;

#& end #&

#& wbetagamma #&

cfw:=function(CT,z,m)
  return 
    ClassFunction(CT,List(ConjugacyClasses(CT),
       c->krondel(Representative(c)^m,z)));
end;

beta:=function(CT,z,m,chi)
  local skp;
  skp:=ScalarProduct(CT,chi,cfw(CT,z,m));
  return SquaredModulus(skp)
         *Size(UnderlyingGroup(CT))
         /DegreeOfCharacter(chi);
end;

gamma:=function(CT,z,m)
  local chi,ichi,clf;
  clf:=List([1..Size(ConjugacyClasses(CT))],i->0);
  for ichi in [1..Size(ConjugacyClasses(CT))] do
    chi:=Irr(CT)[ichi];
    clf:=clf+beta(CT,z,m,chi)*chi;
  od;
  return clf;
end;

#& end #&

#& makedfsi #&

MakeDFSI:=function(GT)
  if IsBound(GT!.CTlist) then return;fi;
  GT!.exp:=Exponent(UnderlyingGroup(GT));
  GT!.div:=DivisorsInt(GT!.exp);
  GT!.CTlist:=[];
  GT!.classposlist:=[];
  GT!.gammalist:=List(ConjugacyClasses(GT),c->[]);
  GT!.mulist:=List(GT!.div,m->[]);
  GT!.mates:=List(ConjugacyClasses(GT),c->[]);
end;

#& end #&

#& providecharactertable #&

providecharactertable:=function(GT,ig)
  if IsBound(GT!.CTlist[ig]) then return;
  fi;
  GT!.CTlist[ig]
    :=OrdinaryCharacterTable(
         Centralizer(
            UnderlyingGroup(GT),
            Representative(
                    ConjugacyClasses(GT)[ig])));
  Irr(GT!.CTlist[ig]);;
  GT!.classposlist[ig]
    :=classpos(GT!.CTlist[ig]);
  return;
end;

#& end #&

#& providegamma #&

providegamma:=function(GT,iz,im)
  local x;
  if IsBound(GT!.gammalist[iz][im]) then return;
  fi;
  GT!.gammalist[iz][im]:=
    gamma(GT!.CTlist[iz],
       Representative(ConjugacyClasses(GT)[iz]),
          GT!.div[im]);
  return;
end;

#& end #&

#& providemate #&

providemate:=function(GT,iz,ih)
    local ig,t;
    if IsBound(GT!.mates[iz][ih]) then return;
    fi; 
    ig:=FusionConjugacyClasses(GT!.CTlist[iz],GT)[ih];
    t:=RepresentativeAction(
            UnderlyingGroup(GT),
            Representative(
               ConjugacyClasses(GT!.CTlist[iz])[ih]),
            Representative(
               ConjugacyClasses(GT)[ig]));
      
    GT!.mates[iz][ih]:=
      (Representative(ConjugacyClasses(GT)[iz])^t)
      ^GT!.classposlist[ig];
end;

#& end #&

#& providemu #&

providemu:=function(GT,ig,im)
  local nmulist,iz,izc,ih,izset,ez,oz,merep,mm,imm,a,c,ihc;
  if IsBound(GT!.mulist[im][ig]) then
    return;
  fi;
  providecharactertable(GT,ig);
  nmulist:=List(ConjugacyClasses(GT!.CTlist[ig]),i->0);
  izset:=Set([1..Size(ConjugacyClasses(GT!.CTlist[ig]))],
             iz->FusionConjugacyClasses(GT!.CTlist[ig],GT)[iz]);
  for iz in izset do
      providecharactertable(GT,iz);
      ez:=Exponent(UnderlyingGroup(GT!.CTlist[iz]));
      oz:=Order(Representative(ConjugacyClasses(GT)[iz]));
      merep:=Gcdex(GT!.div[im],ez);
      mm:=merep.gcd;
      imm:=Position(GT!.div,mm);
      if not RemInt(ez,mm*oz)=0 then continue;
      fi;
      a:=merep.coeff1;
      cc:=Gcdex(a,ez);
      while cc.gcd<>1 do
          a:=a+merep.coeff3;
          cc:=Gcdex(a,ez);
      od;
      c:=cc.coeff1;
      providegamma(GT,iz,imm);
      for ih in [1..Size(ConjugacyClasses(GT!.CTlist[iz]))] do
      if FusionConjugacyClasses(GT!.CTlist[iz],GT)[ih]<>ig 
         then continue;
      fi;
      providemate(GT,iz,ih);
      izc:=GT!.mates[iz][ih];
      ihc:=PowerMap(GT!.CTlist[iz],c,ih);
      nmulist[izc]:=nmulist[izc]+
                    SizesConjugacyClasses(GT!.CTlist[ig])[izc]
                    *GT!.gammalist[iz][imm][ihc];
    od;
  od;
  GT!.mulist[im][ig]:=nmulist;
end;

#& end #&

#& doubleindicator #&

DoubleIndicator:=function(GT,ig,ieta,m)
  local im;    
  MakeDFSI(GT);
  im:=Position(GT!.div,m);
  if ig=1 then 
      return Indicator(GT,m)[ieta];
  fi;
  providemu(GT,ig,im);
  return 1/Size(UnderlyingGroup(GT!.CTlist[ig]))*
         Sum([1..Size(ConjugacyClasses(GT!.CTlist[ig]))],
            iz->Irr(GT!.CTlist[ig])[ieta][iz]
            *GT!.mulist[im][ig][iz]);
end;

#& end #&

\end{lstlisting}
Whenever we need the character table of the centralizer of an element of $G$ (rather, of an element of a cross section of its conjugacy classes), we make sure it is entered into the list by calling:
\begin{lstlisting}[basicstyle=\small,linerange=providecharactertable-end]
#& triv #&

SquaredModulus:=function(z)
  return z*ComplexConjugate(z);
end;

krondel:=function(x,y)
  if x=y then return 1;
  else return 0;
  fi;
end;

classpos:=function(CT)
  return 
    ClassFunction(
       CT,
       [1..Size(ConjugacyClasses(CT))]);
end;

#& end #&

#& wbetagamma #&

cfw:=function(CT,z,m)
  return 
    ClassFunction(CT,List(ConjugacyClasses(CT),
       c->krondel(Representative(c)^m,z)));
end;

beta:=function(CT,z,m,chi)
  local skp;
  skp:=ScalarProduct(CT,chi,cfw(CT,z,m));
  return SquaredModulus(skp)
         *Size(UnderlyingGroup(CT))
         /DegreeOfCharacter(chi);
end;

gamma:=function(CT,z,m)
  local chi,ichi,clf;
  clf:=List([1..Size(ConjugacyClasses(CT))],i->0);
  for ichi in [1..Size(ConjugacyClasses(CT))] do
    chi:=Irr(CT)[ichi];
    clf:=clf+beta(CT,z,m,chi)*chi;
  od;
  return clf;
end;

#& end #&

#& makedfsi #&

MakeDFSI:=function(GT)
  if IsBound(GT!.CTlist) then return;fi;
  GT!.exp:=Exponent(UnderlyingGroup(GT));
  GT!.div:=DivisorsInt(GT!.exp);
  GT!.CTlist:=[];
  GT!.classposlist:=[];
  GT!.gammalist:=List(ConjugacyClasses(GT),c->[]);
  GT!.mulist:=List(GT!.div,m->[]);
  GT!.mates:=List(ConjugacyClasses(GT),c->[]);
end;

#& end #&

#& providecharactertable #&

providecharactertable:=function(GT,ig)
  if IsBound(GT!.CTlist[ig]) then return;
  fi;
  GT!.CTlist[ig]
    :=OrdinaryCharacterTable(
         Centralizer(
            UnderlyingGroup(GT),
            Representative(
                    ConjugacyClasses(GT)[ig])));
  Irr(GT!.CTlist[ig]);;
  GT!.classposlist[ig]
    :=classpos(GT!.CTlist[ig]);
  return;
end;

#& end #&

#& providegamma #&

providegamma:=function(GT,iz,im)
  local x;
  if IsBound(GT!.gammalist[iz][im]) then return;
  fi;
  GT!.gammalist[iz][im]:=
    gamma(GT!.CTlist[iz],
       Representative(ConjugacyClasses(GT)[iz]),
          GT!.div[im]);
  return;
end;

#& end #&

#& providemate #&

providemate:=function(GT,iz,ih)
    local ig,t;
    if IsBound(GT!.mates[iz][ih]) then return;
    fi; 
    ig:=FusionConjugacyClasses(GT!.CTlist[iz],GT)[ih];
    t:=RepresentativeAction(
            UnderlyingGroup(GT),
            Representative(
               ConjugacyClasses(GT!.CTlist[iz])[ih]),
            Representative(
               ConjugacyClasses(GT)[ig]));
      
    GT!.mates[iz][ih]:=
      (Representative(ConjugacyClasses(GT)[iz])^t)
      ^GT!.classposlist[ig];
end;

#& end #&

#& providemu #&

providemu:=function(GT,ig,im)
  local nmulist,iz,izc,ih,izset,ez,oz,merep,mm,imm,a,c,ihc;
  if IsBound(GT!.mulist[im][ig]) then
    return;
  fi;
  providecharactertable(GT,ig);
  nmulist:=List(ConjugacyClasses(GT!.CTlist[ig]),i->0);
  izset:=Set([1..Size(ConjugacyClasses(GT!.CTlist[ig]))],
             iz->FusionConjugacyClasses(GT!.CTlist[ig],GT)[iz]);
  for iz in izset do
      providecharactertable(GT,iz);
      ez:=Exponent(UnderlyingGroup(GT!.CTlist[iz]));
      oz:=Order(Representative(ConjugacyClasses(GT)[iz]));
      merep:=Gcdex(GT!.div[im],ez);
      mm:=merep.gcd;
      imm:=Position(GT!.div,mm);
      if not RemInt(ez,mm*oz)=0 then continue;
      fi;
      a:=merep.coeff1;
      cc:=Gcdex(a,ez);
      while cc.gcd<>1 do
          a:=a+merep.coeff3;
          cc:=Gcdex(a,ez);
      od;
      c:=cc.coeff1;
      providegamma(GT,iz,imm);
      for ih in [1..Size(ConjugacyClasses(GT!.CTlist[iz]))] do
      if FusionConjugacyClasses(GT!.CTlist[iz],GT)[ih]<>ig 
         then continue;
      fi;
      providemate(GT,iz,ih);
      izc:=GT!.mates[iz][ih];
      ihc:=PowerMap(GT!.CTlist[iz],c,ih);
      nmulist[izc]:=nmulist[izc]+
                    SizesConjugacyClasses(GT!.CTlist[ig])[izc]
                    *GT!.gammalist[iz][imm][ihc];
    od;
  od;
  GT!.mulist[im][ig]:=nmulist;
end;

#& end #&

#& doubleindicator #&

DoubleIndicator:=function(GT,ig,ieta,m)
  local im;    
  MakeDFSI(GT);
  im:=Position(GT!.div,m);
  if ig=1 then 
      return Indicator(GT,m)[ieta];
  fi;
  providemu(GT,ig,im);
  return 1/Size(UnderlyingGroup(GT!.CTlist[ig]))*
         Sum([1..Size(ConjugacyClasses(GT!.CTlist[ig]))],
            iz->Irr(GT!.CTlist[ig])[ieta][iz]
            *GT!.mulist[im][ig][iz]);
end;

#& end #&

\end{lstlisting}
The next function provides an element of the list (numbered like the conjugacy classes of $G$) of lists (numbered like the divisors of $\exp(G)$) of the class functions $\gamma_m^z$.
\begin{lstlisting}[basicstyle=\small,linerange=providegamma-end]
#& triv #&

SquaredModulus:=function(z)
  return z*ComplexConjugate(z);
end;

krondel:=function(x,y)
  if x=y then return 1;
  else return 0;
  fi;
end;

classpos:=function(CT)
  return 
    ClassFunction(
       CT,
       [1..Size(ConjugacyClasses(CT))]);
end;

#& end #&

#& wbetagamma #&

cfw:=function(CT,z,m)
  return 
    ClassFunction(CT,List(ConjugacyClasses(CT),
       c->krondel(Representative(c)^m,z)));
end;

beta:=function(CT,z,m,chi)
  local skp;
  skp:=ScalarProduct(CT,chi,cfw(CT,z,m));
  return SquaredModulus(skp)
         *Size(UnderlyingGroup(CT))
         /DegreeOfCharacter(chi);
end;

gamma:=function(CT,z,m)
  local chi,ichi,clf;
  clf:=List([1..Size(ConjugacyClasses(CT))],i->0);
  for ichi in [1..Size(ConjugacyClasses(CT))] do
    chi:=Irr(CT)[ichi];
    clf:=clf+beta(CT,z,m,chi)*chi;
  od;
  return clf;
end;

#& end #&

#& makedfsi #&

MakeDFSI:=function(GT)
  if IsBound(GT!.CTlist) then return;fi;
  GT!.exp:=Exponent(UnderlyingGroup(GT));
  GT!.div:=DivisorsInt(GT!.exp);
  GT!.CTlist:=[];
  GT!.classposlist:=[];
  GT!.gammalist:=List(ConjugacyClasses(GT),c->[]);
  GT!.mulist:=List(GT!.div,m->[]);
  GT!.mates:=List(ConjugacyClasses(GT),c->[]);
end;

#& end #&

#& providecharactertable #&

providecharactertable:=function(GT,ig)
  if IsBound(GT!.CTlist[ig]) then return;
  fi;
  GT!.CTlist[ig]
    :=OrdinaryCharacterTable(
         Centralizer(
            UnderlyingGroup(GT),
            Representative(
                    ConjugacyClasses(GT)[ig])));
  Irr(GT!.CTlist[ig]);;
  GT!.classposlist[ig]
    :=classpos(GT!.CTlist[ig]);
  return;
end;

#& end #&

#& providegamma #&

providegamma:=function(GT,iz,im)
  local x;
  if IsBound(GT!.gammalist[iz][im]) then return;
  fi;
  GT!.gammalist[iz][im]:=
    gamma(GT!.CTlist[iz],
       Representative(ConjugacyClasses(GT)[iz]),
          GT!.div[im]);
  return;
end;

#& end #&

#& providemate #&

providemate:=function(GT,iz,ih)
    local ig,t;
    if IsBound(GT!.mates[iz][ih]) then return;
    fi; 
    ig:=FusionConjugacyClasses(GT!.CTlist[iz],GT)[ih];
    t:=RepresentativeAction(
            UnderlyingGroup(GT),
            Representative(
               ConjugacyClasses(GT!.CTlist[iz])[ih]),
            Representative(
               ConjugacyClasses(GT)[ig]));
      
    GT!.mates[iz][ih]:=
      (Representative(ConjugacyClasses(GT)[iz])^t)
      ^GT!.classposlist[ig];
end;

#& end #&

#& providemu #&

providemu:=function(GT,ig,im)
  local nmulist,iz,izc,ih,izset,ez,oz,merep,mm,imm,a,c,ihc;
  if IsBound(GT!.mulist[im][ig]) then
    return;
  fi;
  providecharactertable(GT,ig);
  nmulist:=List(ConjugacyClasses(GT!.CTlist[ig]),i->0);
  izset:=Set([1..Size(ConjugacyClasses(GT!.CTlist[ig]))],
             iz->FusionConjugacyClasses(GT!.CTlist[ig],GT)[iz]);
  for iz in izset do
      providecharactertable(GT,iz);
      ez:=Exponent(UnderlyingGroup(GT!.CTlist[iz]));
      oz:=Order(Representative(ConjugacyClasses(GT)[iz]));
      merep:=Gcdex(GT!.div[im],ez);
      mm:=merep.gcd;
      imm:=Position(GT!.div,mm);
      if not RemInt(ez,mm*oz)=0 then continue;
      fi;
      a:=merep.coeff1;
      cc:=Gcdex(a,ez);
      while cc.gcd<>1 do
          a:=a+merep.coeff3;
          cc:=Gcdex(a,ez);
      od;
      c:=cc.coeff1;
      providegamma(GT,iz,imm);
      for ih in [1..Size(ConjugacyClasses(GT!.CTlist[iz]))] do
      if FusionConjugacyClasses(GT!.CTlist[iz],GT)[ih]<>ig 
         then continue;
      fi;
      providemate(GT,iz,ih);
      izc:=GT!.mates[iz][ih];
      ihc:=PowerMap(GT!.CTlist[iz],c,ih);
      nmulist[izc]:=nmulist[izc]+
                    SizesConjugacyClasses(GT!.CTlist[ig])[izc]
                    *GT!.gammalist[iz][imm][ihc];
    od;
  od;
  GT!.mulist[im][ig]:=nmulist;
end;

#& end #&

#& doubleindicator #&

DoubleIndicator:=function(GT,ig,ieta,m)
  local im;    
  MakeDFSI(GT);
  im:=Position(GT!.div,m);
  if ig=1 then 
      return Indicator(GT,m)[ieta];
  fi;
  providemu(GT,ig,im);
  return 1/Size(UnderlyingGroup(GT!.CTlist[ig]))*
         Sum([1..Size(ConjugacyClasses(GT!.CTlist[ig]))],
            iz->Irr(GT!.CTlist[ig])[ieta][iz]
            *GT!.mulist[im][ig][iz]);
end;

#& end #&

\end{lstlisting}
To practically apply the formula~\eqref{eq:19}, we need to determine, for $z\in\rCl(G)$ and $h\in\rCl(C_G(z))$, first the class representative $g\in\rCl(G)$ of the conjugacy class of $h\in G$ (this is built into GAP), then an element $t$ such that $t\adhit h=g$, and then the representative $z'\in\rCl(C_G(g))$ of the class of $t\adhit z$ in $C_G(g)$. In other words, the fact that $g$ is conjugate to the element $h$ in the centralizer of $z$ makes that $z$ is conjugate to the element $z'$ in the centralizer of $g$. For lack of a better idea, we call $z'$ the mate of $(z,h)$; it is needed repeatedly, namely for all values of $m$, so it seems worth storing mates:
\begin{lstlisting}[basicstyle=\small,linerange=providemate-end]
#& triv #&

SquaredModulus:=function(z)
  return z*ComplexConjugate(z);
end;

krondel:=function(x,y)
  if x=y then return 1;
  else return 0;
  fi;
end;

classpos:=function(CT)
  return 
    ClassFunction(
       CT,
       [1..Size(ConjugacyClasses(CT))]);
end;

#& end #&

#& wbetagamma #&

cfw:=function(CT,z,m)
  return 
    ClassFunction(CT,List(ConjugacyClasses(CT),
       c->krondel(Representative(c)^m,z)));
end;

beta:=function(CT,z,m,chi)
  local skp;
  skp:=ScalarProduct(CT,chi,cfw(CT,z,m));
  return SquaredModulus(skp)
         *Size(UnderlyingGroup(CT))
         /DegreeOfCharacter(chi);
end;

gamma:=function(CT,z,m)
  local chi,ichi,clf;
  clf:=List([1..Size(ConjugacyClasses(CT))],i->0);
  for ichi in [1..Size(ConjugacyClasses(CT))] do
    chi:=Irr(CT)[ichi];
    clf:=clf+beta(CT,z,m,chi)*chi;
  od;
  return clf;
end;

#& end #&

#& makedfsi #&

MakeDFSI:=function(GT)
  if IsBound(GT!.CTlist) then return;fi;
  GT!.exp:=Exponent(UnderlyingGroup(GT));
  GT!.div:=DivisorsInt(GT!.exp);
  GT!.CTlist:=[];
  GT!.classposlist:=[];
  GT!.gammalist:=List(ConjugacyClasses(GT),c->[]);
  GT!.mulist:=List(GT!.div,m->[]);
  GT!.mates:=List(ConjugacyClasses(GT),c->[]);
end;

#& end #&

#& providecharactertable #&

providecharactertable:=function(GT,ig)
  if IsBound(GT!.CTlist[ig]) then return;
  fi;
  GT!.CTlist[ig]
    :=OrdinaryCharacterTable(
         Centralizer(
            UnderlyingGroup(GT),
            Representative(
                    ConjugacyClasses(GT)[ig])));
  Irr(GT!.CTlist[ig]);;
  GT!.classposlist[ig]
    :=classpos(GT!.CTlist[ig]);
  return;
end;

#& end #&

#& providegamma #&

providegamma:=function(GT,iz,im)
  local x;
  if IsBound(GT!.gammalist[iz][im]) then return;
  fi;
  GT!.gammalist[iz][im]:=
    gamma(GT!.CTlist[iz],
       Representative(ConjugacyClasses(GT)[iz]),
          GT!.div[im]);
  return;
end;

#& end #&

#& providemate #&

providemate:=function(GT,iz,ih)
    local ig,t;
    if IsBound(GT!.mates[iz][ih]) then return;
    fi; 
    ig:=FusionConjugacyClasses(GT!.CTlist[iz],GT)[ih];
    t:=RepresentativeAction(
            UnderlyingGroup(GT),
            Representative(
               ConjugacyClasses(GT!.CTlist[iz])[ih]),
            Representative(
               ConjugacyClasses(GT)[ig]));
      
    GT!.mates[iz][ih]:=
      (Representative(ConjugacyClasses(GT)[iz])^t)
      ^GT!.classposlist[ig];
end;

#& end #&

#& providemu #&

providemu:=function(GT,ig,im)
  local nmulist,iz,izc,ih,izset,ez,oz,merep,mm,imm,a,c,ihc;
  if IsBound(GT!.mulist[im][ig]) then
    return;
  fi;
  providecharactertable(GT,ig);
  nmulist:=List(ConjugacyClasses(GT!.CTlist[ig]),i->0);
  izset:=Set([1..Size(ConjugacyClasses(GT!.CTlist[ig]))],
             iz->FusionConjugacyClasses(GT!.CTlist[ig],GT)[iz]);
  for iz in izset do
      providecharactertable(GT,iz);
      ez:=Exponent(UnderlyingGroup(GT!.CTlist[iz]));
      oz:=Order(Representative(ConjugacyClasses(GT)[iz]));
      merep:=Gcdex(GT!.div[im],ez);
      mm:=merep.gcd;
      imm:=Position(GT!.div,mm);
      if not RemInt(ez,mm*oz)=0 then continue;
      fi;
      a:=merep.coeff1;
      cc:=Gcdex(a,ez);
      while cc.gcd<>1 do
          a:=a+merep.coeff3;
          cc:=Gcdex(a,ez);
      od;
      c:=cc.coeff1;
      providegamma(GT,iz,imm);
      for ih in [1..Size(ConjugacyClasses(GT!.CTlist[iz]))] do
      if FusionConjugacyClasses(GT!.CTlist[iz],GT)[ih]<>ig 
         then continue;
      fi;
      providemate(GT,iz,ih);
      izc:=GT!.mates[iz][ih];
      ihc:=PowerMap(GT!.CTlist[iz],c,ih);
      nmulist[izc]:=nmulist[izc]+
                    SizesConjugacyClasses(GT!.CTlist[ig])[izc]
                    *GT!.gammalist[iz][imm][ihc];
    od;
  od;
  GT!.mulist[im][ig]:=nmulist;
end;

#& end #&

#& doubleindicator #&

DoubleIndicator:=function(GT,ig,ieta,m)
  local im;    
  MakeDFSI(GT);
  im:=Position(GT!.div,m);
  if ig=1 then 
      return Indicator(GT,m)[ieta];
  fi;
  providemu(GT,ig,im);
  return 1/Size(UnderlyingGroup(GT!.CTlist[ig]))*
         Sum([1..Size(ConjugacyClasses(GT!.CTlist[ig]))],
            iz->Irr(GT!.CTlist[ig])[ieta][iz]
            *GT!.mulist[im][ig][iz]);
end;

#& end #&

\end{lstlisting}
The next function computes $\mu_m(g)$, if it is not yet known, where $m$ and $g$ are represented by their positions in the lists of divisors of $\exp(G)$, resp.\ the list of conjugacy classes of $G$. 
\begin{lstlisting}[basicstyle=\tiny,linerange=providemu-end]
#& triv #&

SquaredModulus:=function(z)
  return z*ComplexConjugate(z);
end;

krondel:=function(x,y)
  if x=y then return 1;
  else return 0;
  fi;
end;

classpos:=function(CT)
  return 
    ClassFunction(
       CT,
       [1..Size(ConjugacyClasses(CT))]);
end;

#& end #&

#& wbetagamma #&

cfw:=function(CT,z,m)
  return 
    ClassFunction(CT,List(ConjugacyClasses(CT),
       c->krondel(Representative(c)^m,z)));
end;

beta:=function(CT,z,m,chi)
  local skp;
  skp:=ScalarProduct(CT,chi,cfw(CT,z,m));
  return SquaredModulus(skp)
         *Size(UnderlyingGroup(CT))
         /DegreeOfCharacter(chi);
end;

gamma:=function(CT,z,m)
  local chi,ichi,clf;
  clf:=List([1..Size(ConjugacyClasses(CT))],i->0);
  for ichi in [1..Size(ConjugacyClasses(CT))] do
    chi:=Irr(CT)[ichi];
    clf:=clf+beta(CT,z,m,chi)*chi;
  od;
  return clf;
end;

#& end #&

#& makedfsi #&

MakeDFSI:=function(GT)
  if IsBound(GT!.CTlist) then return;fi;
  GT!.exp:=Exponent(UnderlyingGroup(GT));
  GT!.div:=DivisorsInt(GT!.exp);
  GT!.CTlist:=[];
  GT!.classposlist:=[];
  GT!.gammalist:=List(ConjugacyClasses(GT),c->[]);
  GT!.mulist:=List(GT!.div,m->[]);
  GT!.mates:=List(ConjugacyClasses(GT),c->[]);
end;

#& end #&

#& providecharactertable #&

providecharactertable:=function(GT,ig)
  if IsBound(GT!.CTlist[ig]) then return;
  fi;
  GT!.CTlist[ig]
    :=OrdinaryCharacterTable(
         Centralizer(
            UnderlyingGroup(GT),
            Representative(
                    ConjugacyClasses(GT)[ig])));
  Irr(GT!.CTlist[ig]);;
  GT!.classposlist[ig]
    :=classpos(GT!.CTlist[ig]);
  return;
end;

#& end #&

#& providegamma #&

providegamma:=function(GT,iz,im)
  local x;
  if IsBound(GT!.gammalist[iz][im]) then return;
  fi;
  GT!.gammalist[iz][im]:=
    gamma(GT!.CTlist[iz],
       Representative(ConjugacyClasses(GT)[iz]),
          GT!.div[im]);
  return;
end;

#& end #&

#& providemate #&

providemate:=function(GT,iz,ih)
    local ig,t;
    if IsBound(GT!.mates[iz][ih]) then return;
    fi; 
    ig:=FusionConjugacyClasses(GT!.CTlist[iz],GT)[ih];
    t:=RepresentativeAction(
            UnderlyingGroup(GT),
            Representative(
               ConjugacyClasses(GT!.CTlist[iz])[ih]),
            Representative(
               ConjugacyClasses(GT)[ig]));
      
    GT!.mates[iz][ih]:=
      (Representative(ConjugacyClasses(GT)[iz])^t)
      ^GT!.classposlist[ig];
end;

#& end #&

#& providemu #&

providemu:=function(GT,ig,im)
  local nmulist,iz,izc,ih,izset,ez,oz,merep,mm,imm,a,c,ihc;
  if IsBound(GT!.mulist[im][ig]) then
    return;
  fi;
  providecharactertable(GT,ig);
  nmulist:=List(ConjugacyClasses(GT!.CTlist[ig]),i->0);
  izset:=Set([1..Size(ConjugacyClasses(GT!.CTlist[ig]))],
             iz->FusionConjugacyClasses(GT!.CTlist[ig],GT)[iz]);
  for iz in izset do
      providecharactertable(GT,iz);
      ez:=Exponent(UnderlyingGroup(GT!.CTlist[iz]));
      oz:=Order(Representative(ConjugacyClasses(GT)[iz]));
      merep:=Gcdex(GT!.div[im],ez);
      mm:=merep.gcd;
      imm:=Position(GT!.div,mm);
      if not RemInt(ez,mm*oz)=0 then continue;
      fi;
      a:=merep.coeff1;
      cc:=Gcdex(a,ez);
      while cc.gcd<>1 do
          a:=a+merep.coeff3;
          cc:=Gcdex(a,ez);
      od;
      c:=cc.coeff1;
      providegamma(GT,iz,imm);
      for ih in [1..Size(ConjugacyClasses(GT!.CTlist[iz]))] do
      if FusionConjugacyClasses(GT!.CTlist[iz],GT)[ih]<>ig 
         then continue;
      fi;
      providemate(GT,iz,ih);
      izc:=GT!.mates[iz][ih];
      ihc:=PowerMap(GT!.CTlist[iz],c,ih);
      nmulist[izc]:=nmulist[izc]+
                    SizesConjugacyClasses(GT!.CTlist[ig])[izc]
                    *GT!.gammalist[iz][imm][ihc];
    od;
  od;
  GT!.mulist[im][ig]:=nmulist;
end;

#& end #&

#& doubleindicator #&

DoubleIndicator:=function(GT,ig,ieta,m)
  local im;    
  MakeDFSI(GT);
  im:=Position(GT!.div,m);
  if ig=1 then 
      return Indicator(GT,m)[ieta];
  fi;
  providemu(GT,ig,im);
  return 1/Size(UnderlyingGroup(GT!.CTlist[ig]))*
         Sum([1..Size(ConjugacyClasses(GT!.CTlist[ig]))],
            iz->Irr(GT!.CTlist[ig])[ieta][iz]
            *GT!.mulist[im][ig][iz]);
end;

#& end #&

\end{lstlisting}
Finally the function \lstinline!DoubleIndicator! computes an indicator of a module of the Drinfeld double $D(G)$, given a character table \lstinline!GT! with underlying group $G$, the class position of an element $g\in G$, the number of an irreducible character $\eta\in\Irr(C_G(g))$, and a number $m$ which is supposed to divide $\exp(G)$. It begins by making sure that the appropriate element $\mu_m(g)$ is provided.
\begin{lstlisting}[basicstyle=\tiny,linerange=doubleindicator-end]
#& triv #&

SquaredModulus:=function(z)
  return z*ComplexConjugate(z);
end;

krondel:=function(x,y)
  if x=y then return 1;
  else return 0;
  fi;
end;

classpos:=function(CT)
  return 
    ClassFunction(
       CT,
       [1..Size(ConjugacyClasses(CT))]);
end;

#& end #&

#& wbetagamma #&

cfw:=function(CT,z,m)
  return 
    ClassFunction(CT,List(ConjugacyClasses(CT),
       c->krondel(Representative(c)^m,z)));
end;

beta:=function(CT,z,m,chi)
  local skp;
  skp:=ScalarProduct(CT,chi,cfw(CT,z,m));
  return SquaredModulus(skp)
         *Size(UnderlyingGroup(CT))
         /DegreeOfCharacter(chi);
end;

gamma:=function(CT,z,m)
  local chi,ichi,clf;
  clf:=List([1..Size(ConjugacyClasses(CT))],i->0);
  for ichi in [1..Size(ConjugacyClasses(CT))] do
    chi:=Irr(CT)[ichi];
    clf:=clf+beta(CT,z,m,chi)*chi;
  od;
  return clf;
end;

#& end #&

#& makedfsi #&

MakeDFSI:=function(GT)
  if IsBound(GT!.CTlist) then return;fi;
  GT!.exp:=Exponent(UnderlyingGroup(GT));
  GT!.div:=DivisorsInt(GT!.exp);
  GT!.CTlist:=[];
  GT!.classposlist:=[];
  GT!.gammalist:=List(ConjugacyClasses(GT),c->[]);
  GT!.mulist:=List(GT!.div,m->[]);
  GT!.mates:=List(ConjugacyClasses(GT),c->[]);
end;

#& end #&

#& providecharactertable #&

providecharactertable:=function(GT,ig)
  if IsBound(GT!.CTlist[ig]) then return;
  fi;
  GT!.CTlist[ig]
    :=OrdinaryCharacterTable(
         Centralizer(
            UnderlyingGroup(GT),
            Representative(
                    ConjugacyClasses(GT)[ig])));
  Irr(GT!.CTlist[ig]);;
  GT!.classposlist[ig]
    :=classpos(GT!.CTlist[ig]);
  return;
end;

#& end #&

#& providegamma #&

providegamma:=function(GT,iz,im)
  local x;
  if IsBound(GT!.gammalist[iz][im]) then return;
  fi;
  GT!.gammalist[iz][im]:=
    gamma(GT!.CTlist[iz],
       Representative(ConjugacyClasses(GT)[iz]),
          GT!.div[im]);
  return;
end;

#& end #&

#& providemate #&

providemate:=function(GT,iz,ih)
    local ig,t;
    if IsBound(GT!.mates[iz][ih]) then return;
    fi; 
    ig:=FusionConjugacyClasses(GT!.CTlist[iz],GT)[ih];
    t:=RepresentativeAction(
            UnderlyingGroup(GT),
            Representative(
               ConjugacyClasses(GT!.CTlist[iz])[ih]),
            Representative(
               ConjugacyClasses(GT)[ig]));
      
    GT!.mates[iz][ih]:=
      (Representative(ConjugacyClasses(GT)[iz])^t)
      ^GT!.classposlist[ig];
end;

#& end #&

#& providemu #&

providemu:=function(GT,ig,im)
  local nmulist,iz,izc,ih,izset,ez,oz,merep,mm,imm,a,c,ihc;
  if IsBound(GT!.mulist[im][ig]) then
    return;
  fi;
  providecharactertable(GT,ig);
  nmulist:=List(ConjugacyClasses(GT!.CTlist[ig]),i->0);
  izset:=Set([1..Size(ConjugacyClasses(GT!.CTlist[ig]))],
             iz->FusionConjugacyClasses(GT!.CTlist[ig],GT)[iz]);
  for iz in izset do
      providecharactertable(GT,iz);
      ez:=Exponent(UnderlyingGroup(GT!.CTlist[iz]));
      oz:=Order(Representative(ConjugacyClasses(GT)[iz]));
      merep:=Gcdex(GT!.div[im],ez);
      mm:=merep.gcd;
      imm:=Position(GT!.div,mm);
      if not RemInt(ez,mm*oz)=0 then continue;
      fi;
      a:=merep.coeff1;
      cc:=Gcdex(a,ez);
      while cc.gcd<>1 do
          a:=a+merep.coeff3;
          cc:=Gcdex(a,ez);
      od;
      c:=cc.coeff1;
      providegamma(GT,iz,imm);
      for ih in [1..Size(ConjugacyClasses(GT!.CTlist[iz]))] do
      if FusionConjugacyClasses(GT!.CTlist[iz],GT)[ih]<>ig 
         then continue;
      fi;
      providemate(GT,iz,ih);
      izc:=GT!.mates[iz][ih];
      ihc:=PowerMap(GT!.CTlist[iz],c,ih);
      nmulist[izc]:=nmulist[izc]+
                    SizesConjugacyClasses(GT!.CTlist[ig])[izc]
                    *GT!.gammalist[iz][imm][ihc];
    od;
  od;
  GT!.mulist[im][ig]:=nmulist;
end;

#& end #&

#& doubleindicator #&

DoubleIndicator:=function(GT,ig,ieta,m)
  local im;    
  MakeDFSI(GT);
  im:=Position(GT!.div,m);
  if ig=1 then 
      return Indicator(GT,m)[ieta];
  fi;
  providemu(GT,ig,im);
  return 1/Size(UnderlyingGroup(GT!.CTlist[ig]))*
         Sum([1..Size(ConjugacyClasses(GT!.CTlist[ig]))],
            iz->Irr(GT!.CTlist[ig])[ieta][iz]
            *GT!.mulist[im][ig][iz]);
end;

#& end #&

\end{lstlisting}

\section{Sample runs}
\label{sec:applications}

One interesting test case for our algorithm was the computation of the indicator values for the doubles of $D(S_n)$. Scharf \cite{MR1113784} proved that the higher indicator of all ordinary representations of the symmetric groups themselves are positive integers.  Courter \cite{MR3103664} investigated the analogous conjecture for the representations of the doubles of symmetric groups (where Iovanov, Mason, and Montgomery \cite{MR3275646} show that the indicators are in fact integers), and was able to verify the conjecture up to $S_{10}$. She reports having taken a week of computer time to finish this largest case, and that $S_{11}$ was too heavy a task for her hardware. We do not have precise information on that hardware, but should of course assume that today's hardware is considerably more powerful. The present author had a laptop running under linux with a quad core CPU ({\tt Intel(R) Core(TM) i7-3520M CPU @ 2.90GHz} according to the system's internal information) and 8GB of RAM. On this machine, code borrowed directly from Courter will deal with the indicators for the double of $S_{10}$ in 8 hours and 25 minutes. (More precisely, this is the time taken to compute the elements $\mu_m(g)$ for all nontrivial divisors of $\exp(G)$ and $g$ from a system of representatives of the conjugacy classes of $G$, but $g\neq e$. No character tables of centralizers were computed, hence no actual indicator values.) The code presented in \cref{sec:implementation}, on the other hand, computes all the indicator values $\nu_m(g,\eta)$ for the double of the symmetric group $S_{10}$, again excepting $m=1$ and $m=\exp(G)$, in about one minute. The comparison was greatly facilitated by the fact that Rebecca Courter was kind enough to send  all her code and computation results; also, having this material was very helpful in early stages of the author's clumsy programming, because direct comparison of the results for small examples helped eliminate errors. We did not make a real effort to investigate Courter's code for possible improvements; however, its limitations for increasing size of the group seem evident: A key to speeding up computations is that the program stores a list of $m$-th powers of all the group elements to avoid recalculations. The length of this list grows more than exponentially with the degree of the symmetric group. The order of the symmetric group on $18$ letters is $6,402,373,705,728,000\approx 6\cdot 10^{15}$ and thus more than a hundred times the square of the order of the largest group $S_{10}$ treated in \cite{MR3103664}. To brag in a different way: More than 10 petabytes of storage would be needed, or, if we were to do without the speedup, some $10^{15}$ elements would have to be looped through many times. Our algorithms, on the other hand, only loop through lists of conjugacy classes, in the order of up to a thousand elements; even with nested loops, no more than millions of iterations are needed to deal with $S_{18}$. This relies quite shamelessly on the ability of GAP to also deal with the calculation of centralizer subgroups, conjugacy classes, and character tables in an efficient way (not, it seems, ever looping through all the group elements). Note also that we rely on GAP to decide whether two given elements are conjugate in $G$, and if they are, to find an element conjugating one into the other. If we mentioned the symmetric group on 18 letters above, it is because it is the largest symmetric group for which the code from \cref{sec:implementation} has been able to compute all the indicators of the double (putting in a little less than $12$ hours of CPU time according to GAP, a little under $7$ hours being spent in computing character tables). None of them were negative. For the symmetric group on $19$ letters, the author's laptop stalled, during the attempt to find the irreducible characters of one of the centralizers.

If we want to push ahead in the quest for a counterexample (to nonnegativity of the indicators), there are various ways to go beyond $S_{18}$. The easiest is to turn to more muscled hardware, and we admit with some embarrassment having used the computing power provided by the \emph{Centre de Calcul} of the \emph{Université de Bourgogne} to treat $S_{19}$. More intelligently, one could use the fact that the centralizers of elements of the symmetric group are known: They are direct products of wreath products of symmetric and cyclic groups. It should be possible to put this information to use and thus make GAP compute the character tables of centralizers in a more intelligent way without appealing blindly to the Dixon-Schneider algorithm it usually uses, or to break down the class functions $\gamma_m^z$ along this decomposition of the centralizers. As it stands, the conjecture motivated calculations that, by their size, demonstrate the efficiency of the algorithm in its form described in \cref{sec:implementation}, which is valid for any group.

Another interesting test case was suggested by \cite{MR3275646}: Using the iterator providing finite simple groups, we have set GAP on the task of calculating the indicators of the irreducible modules of the doubles of simple groups (excluding the groups $\PSL_2(q)$). This worked up to orders below $200,000,000$, whereafter the author's laptop stalled while trying to calculate a character table. The original motivation for this quest was the question whether all the indicators for the doubles of simple groups are integers (the indicators we computed were), but as it will turn out in \cref{sec:integr-quest}, that question can be attacked without actually calculating the indicators. 

\section{Rationality questions and $FSZ$-groups}
\label{sec:integr-quest}

Recall that Iovanov, Mason, and Montgomery \cite{MR3275646} call a finite group $G$ an an $FSZ_m$-group if for every simple $D(G)$-module $V$ we have $\nu_m(V)\in\Z$, and an $FSZ$-group if it is an $FSZ_m$-group for every $m$ (dividing the exponent of $G$). In general, it is known that the indicators are algebraic integers in the cyclotomic field $\CF{\exp(G)}$, and that they are real. Moreover, the $m$-th indicators $\nu_m(V)$ lie in the cyclotomic field $\CF m$. Instead of asking whether all indicators for a given group double are (rational) integers, one may of course ask which cyclotomic field, possibly smaller than prescribed by $m$ and $\exp(G)$ alone, they belong to.

By and large, our results on the rationality of indicators stem from the behavior of the $\beta$ and $\gamma$ maps under Adams operators; we begin with the following:
\begin{Lem}\label{prop:6}
  Let $G$ be a finite group, and $(r,\exp(G))=1$.
  \begin{enumerate}
  \item \label{item:7}$\nu_m(g^r,\eta)=\nu_m(g,\adams r\eta)=\sigma_r\nu_m(g,\eta)$ for each $g\in G$ and $\eta\in\Irr(C_G(g))$; in particular $\{\nu_m(g,\eta)|\eta\in \Irr(C_G(g))\}$ depends only on the rational conjugacy class of $g\in G$.
  \item \label{item:10}$\beta_m(z^r,\chi)=\beta_m(z,\adams r\chi)=\sigma_r\beta_m(z,\chi)$ for each $z\in G$ and $\chi\in\Irr(C_G(z))$; in particular $\{\beta_m(z,\chi)|\chi\in \Irr(C_G(z))\}$ depends only on the rational conjugacy class of $z\in G$.
  \end{enumerate}
\end{Lem}
\begin{proof}
  In fact the formula in \cref{item:7} was already proved in \cite{2015arXiv150202902S}; it can also be deduced from (and is equivalent to) some of the properties of the class functions gamma proved in \cref{prop:3}: By \eqref{eq:13} we have, for $rs\equiv 1\mod \exp(G)$:
  \begin{multline*}
    |C_G(g)|\nu_m(g^r,\eta)=
    \sum_{z\in C_G(g)}\adams r\gamma_m^z(g)\eta(z)=\sum_{z\in C_G(g)}\gamma_m^{z^s}(g)\eta(z)=\\=\sum_{z\in C_G(g)}\gamma_m^z(g)\eta(z^r)=|C_G(g)|\nu_m(g,\adams r\eta)=|C_G(g)|\sigma_r\nu_m(g,\eta)
  \end{multline*}
  since $\eta(z^r)=\adams r\eta(z)=\sigma_r\eta(z)$. The calculation
  \begin{multline*}
    |\langle\chi,w_m^{z^r}\rangle|%
    =|\langle\chi,\adams sw_m^z\rangle|%
    =|\langle\adams r\chi,w_m^z\rangle|
    =|\sigma_r\langle\chi,w_m^z\rangle|
    =\sigma_r|\langle\chi,w_m^z\rangle|
  \end{multline*}
  proves the corresponding formula in \cref{item:10}.
\end{proof}

By \cite[Cor.3.2]{MR3275646}, $G$ is an $FSZ_m$-group if and only if, for every $r$ coprime to the exponent of $G$ one has $|G_m(g,z)|=|G_m(g,z^r)|$. The same arguments used there show in fact that this condition is satisfied for all $r$ such that $(r,\exp(G))=1$ and $r\equiv 1\mod d$ if and only if $\nu_m(g,\eta)\in\CF d$ for all $g$ and $\eta$.
\begin{Cor}
  Let $G$ be a finite group and $d|\exp(G)$. The following are equivalent:
  \begin{enumerate}
  \item $\nu_m(g,\eta)\in\CF d$ for all $g\in G$ and $\eta\in\Irr(C_G(g))$.
  \item $\adams r\gamma_m^z=\gamma_m^z$ for all $z\in G$ and all integers $r$ satisfying $(r,\exp(G))=1$ and $r\equiv 1\mod d$.
  \end{enumerate}
\end{Cor}
\begin{proof}
  The condition from \cite{MR3275646} asks for $\gamma_m^z=\gamma_m^{z^r}$ for $r$ in 
$(\mathbb Z_m)^*$; let $s$ be the inverse of $r$ and recall that we have shown in \cref{prop:3} that $\adams s\gamma_m^{z^s}=\gamma_m^z$.
\end{proof}

An consequence of the above rationality criterion is a generalization of (the less interesting part of) results from \cite{MR1919158,MR2491893}. To wit, these papers show how to calculate the second indicators of the $D(G)$-module associated to a character $\eta$ of $C_G(g)$ in terms of the indicators of $\eta$ and the character induced from the centralizer to a certain normalizer. We will fall considerably short of such an explicit result, but will at least generalize the fact that the indicator only depends on the induced character.

Recall that the normalizer of an element $g\in G$ is $N_G(g)=N_G(\langle g\rangle)$, the normalizer of the cyclic subgroup generated by $g$. Thus $t\in N_G(g)$ iff $t\adhit g=g^r$ with $(r,o(g))=1$; in fact we can choose $r$ such that $(r,\exp(G))=1$. We will use the following ``restricted normalizer'': For an integer $d$ let
\begin{equation*}
  N_G^d(g):=\{t\in G|t\adhit g=g^r\text{ with }(r,\exp(G))=1\text{ and }r\equiv 1\mod d\},
\end{equation*}
so that in particular $N_G^1(g)=N_G(g)$. Another special case is
\begin{equation*}
  N_G^2(g)=\{t\in G|t\adhit g\in\{g,g\inv\}\},
\end{equation*}
which would be called an ``extended stabilizer'' in \cite{MR1919158} and the normalizer of the set $\{g,g\inv\}$ in \cite{MR2491893}. In these papers, it is shown how to calculate the second indicator of the object associated to $\eta\in \Irr(C_G(g))$ in terms of the character of $N_G^2(g)$ induced from $\eta$. 

Note that the centralizer $C_G(g)$ is normal in the normalizer $N_G(g)$. Thus, by Clifford theory, $\eta,\eta'\in\Irr(C_G(g))$ induce up to the same character of $N_G^d(g)$ if and only if they are conjugate under the action of $N_G^d(g)$. 
\begin{Prop}
  Let $G$ be a finite group , $g\in G$, and $d|\exp(G)$. Assume that $\nu_m(g,\eta)\in\CF d$ for all $\eta\in\Irr(C_G(g))$. Then $\nu_m(g,\eta)$ depends only on the induced character $\Ind_{C_G(g)}^{N_G^d(g)}(\eta)$; equivalently, $\nu_m(g,\eta)$ is invariant under the action of $N_G^d(g)$ on the characters of $C_G(g)$.
\end{Prop}
\begin{proof}
  Let $t\in N_G^d(g)$ and $t\inv\adhit g=g ^r$ with $(r,\exp(G))=1$ and $r\equiv 1\mod d$. Then $\nu_m(g,t\adhit\eta)=\nu_m(t\inv\adhit g,\eta)=\nu_m(g^r,\eta)$ because the two couples $(g,t\adhit\eta)$ and $(t\inv\adhit g,\eta)$ induce isomorphic $D(G)$-modules. By assumption $\nu_m(g,\eta)=\sigma_r\nu_m(g,\eta)$ and so the formula reviewed in \cref{prop:6}~\cref{item:7} proves the claim.
\end{proof}

Next, we use the same reasoning as in \cite{MR3275646} to play the behavior of $\gamma$ under Adams operators back to the arithmetic properties of the coefficients $\beta$:

\begin{Prop}\label{prop:4}
  We have $\nu_m(g,\eta)\in\CF d$ for all $g\in G$ and $\eta\in\Irr(C_G(g))$ if and only if $\beta_m(z,\chi)\in\CF d$ for all $z\in G$ and $\chi\in\Irr(C_G(z))$. 
\end{Prop}
\begin{proof}
  $\gamma_m^z$ is a rational valued class function. Thus by \cite[Thm.25]{MR543841} it is a linear combination of characters with coefficients in $\CF d$ if and only if $\adams r\gamma_m^z=\gamma_m^z$ for all $r$ with $(r,\exp(G))=1$ and $r\equiv 1\mod d$.
\end{proof}
While the proposition shows that for the double $D(G)$ irrational indicator values occur if and only if irrational values of $\beta_m$ occur, we shall see in an example below that this will not necessarily happen for the same conjugacy classes, although we note:
\begin{Lem}
  If $\beta_m(z,\chi)\not\in\CF d$ for some $z\in G$ and $\chi\in\Irr(C_G(z))$, then there is a $D(C_G(z))$-module $V$ with $\nu_m(V)\not\in\CF d$.
\end{Lem}
\begin{proof}
  To write the argument properly, we have to introduce the more precise notation $\beta_m^G(z,\chi):=\beta_m(z,\chi)$ to take into account that the centralizer $C_G(z)$ depends on the ambient group, not only on the element $z$. That said, we have $\beta_m^G(z,\chi)=\beta_m^{C_G(z)}(z,\chi)$ since $z$ is central in $C_G(z)$. Thus a module of $D(C_G(z))$ with indicator not in $\CF d$ exists by \cref{prop:4}.
\end{proof}

\begin{Lem}\label{prop:5}
  Let $G$ be a finite group, $z\in G$, $\chi\in\Irr(C_G(z))$, and $m|\exp(C_G(z))/o(z)$. Then $\beta_m(z,\chi)\in\CF{(o(z),m)}$. In particular, if $(o(z),m)\in\{1,2,3,4,6\}$, then $\beta_m(z,\chi)\in\mathbb Q$.
\end{Lem}
\begin{proof}
  If $r\equiv 1\mod o(z)$ then $\adams r\gamma_m^z=\adams r\gamma_m^{z^r}=\gamma_m^z$ showing $\beta_m(z,\chi)\in\CF{o(z)}$. Since $m$-th indicators are in $\CF m\cap\mathbb R$, we are done.
\end{proof}

\Cref{prop:4} gives a criterion for deciding whether all the higher indicators for the double of a given group are rational, particularly efficient if we use \cref{prop:5} to skip combinations of $m$ and $z$ that are useless to check (which may happen because $o(z)$ is ``too small'' as it needs to have a ``large'' divisor, or because $o(z)$ is ``too large'' because $\exp(C_G(z))/o(z)$ needs to have a ``large'' divisor). We note that the result is related to, but stronger than, results relating the $FSZ_+$-property of a group to the $FSZ$-property of its centralizer subgroups in \cite{MR3275646}. A GAP function performing the test is:

\begin{lstlisting}
FSZtest:=function(G)
    local C,CT,z,cl,div,chi;
    for cl in RationalClasses(G) do
        z:=Representative(cl);
        if Order(z) in [1,2,3,4,6] then continue;
        fi;
        C:=Centralizer(G,z);
        CT:=OrdinaryCharacterTable(C);
        div:=DivisorsInt(Exponent(C)/Order(z));
        div:=Filtered(div,m-> not m in [1,2,3,4,6]);
        for m in div do
            d:=Gcd(m,Order(z));
            if d in [1,2,3,4,6] then continue;
            fi;
            for chi in Irr(CT) do
                if not IsRat(beta(CT,z,m,chi)) then 
                    Print(Order(C),",",m,"\n");
                    return false;
                fi;
            od;
        od;
    od;
    return true;    
end;

\end{lstlisting}

We used it to test simple groups as provided to us by the iterator \lstinline!SimpleGroupsIterator!, skipping the groups $\PSL_2(q)$ for which \cite{MR3275646} already showed that they are $FSZ$ (but not caring to skip other groups treated there). We can report that the exceptional Chevalley group $G_2(5)$ of order $5^6(5^6-1)(5^2-1)=5,859,000,000$ is the only counterexample whose double affords a representation with an irrational (fifth) FS-indicator among the simple groups up to order
$|\PSL(3,29)|=499,631,102,880$ (where we turned to the stronger hardware setup provided by the \emph{Centre de Calcul} of the \emph{Université de Bourgogne} after treating $\PSL(3,27)$ of order $282,027,786,768$ on the above-mentioned laptop.)

While the code included above can report rather swiftly that $G_2(5)$ is not $FSZ$, the code from \cref{sec:implementation} fails to compute the indicators. The reason for this, however, is that we asked GAP to compute the character tables of all the centralizers by its standard methods. This is no problem for all the \emph{nontrivial} centralizers, which have relatively  harmless orders. It is the character table of $G_2(5)$ itself that stalls the computation. Now $G_2(5)$ is an ``atlas group'' whose character table is available from the atlas of finite groups \cite{MR827219}; to the GAP user, it is of course more practical to invoke the character table library package \cite{CTblLib1.2.1} that will load the character table of the group $G_2(5)$ from a database. Since the atlas tables in GAP do not initially know their underlying groups, one has to use the \lstinline!CharacterTableWithStoredGroup! function to explicitly ask for the connection of the table with the group, and this done one has to explicitly enter this character table as the table of the centralizer of the neutral element in the data structures described in \cref{sec:implementation}, before leaving all the rest of the calculation (along with the calculation of the character tables of all the other centralizers) to the code given in that section, which will only take a little more than two minutes of laptop CPU time to complete its job. We summarize some of the results in \cref{fig:indival}, rewriting the integer linear combinations of fifth roots of unity $\zeta_5^k$ output by GAP as elements of $\mathbb Q(\sqrt 5)=\mathbb R\cap \CF 5$, e.g.\ $-60(\zeta_5+\zeta_5^4)-10(\zeta_5^2+\zeta_5^3)=35-25\sqrt 5$
\begin{table}[t]
  \centering
\renewcommand\arraystretch{1.25}
  \begin{tabular}{c||c|c|c|c|c|c}
    $g$ in &$C_G(g)\cong$&$\nu_5(g,\eta)$&$\nu_6(g,\eta)$&multiplicity
    \\\hline\hline
\multirow{7}{*}{10b}&\multirow{7}{*}{\begin{tabular}{c}order\\$600$\\with\\$45$\\conjugacy\\classes\end{tabular}}&81&10621& 5
    \\\cline{3-5}
    % -60(\zeta_5+\zeta_5^4)-10(\zeta_5^2+\zeta_5^3)=
           &&$35\pm 25\sqrt 5$&21163&5+5
    \\\cline{3-5}
&&%$\begin{array}{r}-141(\zeta_5+\zeta_5^4)-91(\zeta_5^2+\zeta_5^3)\\=116-50\sqrt 5\end{array}$
           $116\pm 50\sqrt 5$&31783&5+5\\\cline{3-5}
       &&$70$&$42326$& 5\\\cline{3-5}
&&$ 72$&$ 42330$& 5\\\cline{3-5}
&&$ 152$&$ 52951$& 5\\\cline{3-5}
&&$ 230$&$ 63561$& 5\\\hline\hline
    %\multirow{9}{*}{10b}&\multirow{9}{*}{\begin{tabular}{c}order\\$600$\\with\\$45$\\conjugacy\\classes\end{tabular}}&81&10621& 5
%     \\\cline{3-5}
%     % -60(\zeta_5+\zeta_5^4)-10(\zeta_5^2+\zeta_5^3)=
%            &&$35-25\sqrt 5$&21163&5
%     \\\cline{3-5}
%            &&%$\begin{array}{r}-10(\zeta_5+\zeta_5^4)-60(\zeta_5^2+\zeta_5^3)\\=35+25\sqrt 5\end{array}$
%            $35+25\sqrt 5$&21163&5\\\cline{3-5}
%            &&%$\begin{array}{r}-141(\zeta_5+\zeta_5^4)-91(\zeta_5^2+\zeta_5^3)\\=116-50\sqrt 5\end{array}$
%            $116-50\sqrt 5$&31783&5\\\cline{3-5}
%            &&%$\begin{array}{r}-91(\zeta_5+\zeta_5^4)-141(\zeta_5^2+\zeta_5^3)\\=116+50\sqrt 5\end{array}$
%     $116+50\sqrt 5$&31783&5\\\cline{3-5}
% &&$70$&$42326$& 5\\\cline{3-5}
% &&$ 72$&$ 42330$& 5\\\cline{3-5}
% &&$ 152$&$ 52951$& 5\\\cline{3-5}
% &&$ 230$&$ 63561$& 5\\\hline
  % 81, 10621 5
  % dings 21163 5
  % dings , 21163 5
  % dings , 31783 5
  % dings , 31783 5
  % 70, 42326 5
  % 72, 42330 5
  % 152, 52951 5
  % 230, 63561 5
%%%%%%%%%%%%%%%%%%%%%%%%%%5  
% -324*E(5)-274*E(5)^2-274*E(5)^3-324*E(5)^4
    \multirow{4}{*}{10c}&\multirow{4}{*}{$C_{10}\times C_5$}&$299\pm 25\sqrt 5$&119597&  5+5\\\cline{3-5}
    % -326*E(5)-276*E(5)^2-276*E(5)^3-326*E(5)^4
    &&$301\pm 25\sqrt 5$& 119605 &5+5\\\cline{3-5}
    && 424 & 119657 &5\\\cline{3-5}
    && 426 &119665 &5\\\hline\hline

% \multirow{6}{*}{10c}&\multirow{6}{*}{$C_{10}\times C_5$}&$299-25\sqrt 5$&119597&  10\\\cline{3-5}
  %   % -274*E(5)-324*E(5)^2-324*E(5)^3-274*E(5)^4
  %   &&$299+25\sqrt 5$&119597&  10\\\cline{3-5}
  %   % -326*E(5)-276*E(5)^2-276*E(5)^3-326*E(5)^4
  %   &&$301-25\sqrt 5$& 119605 &10\\\cline{3-5}
  % % -276*E(5)-326*E(5)^2-326*E(5)^3-276*E(5)^4  119605 10
  %   &&$301+25\sqrt 5$& 119605 &10\\\cline{3-5}
  %   && 424 & 119657 &5\\\cline{3-5}
  %   && 426 &119665 &5\\\hline
    % -324*E(5)-274*E(5)^2-274*E(5)^3-324*E(5)^4
    \multirow{8}{*}{10d}&\multirow{8}{*}{$C_{10}\times C_5$}
                         &$299\pm 25\sqrt 5$&117091& 8+8\\\cline{3-5}
                         % -324*E(5)-274*E(5)^2-274*E(5)^3-324*E(5)^4
    &&$299\pm 25\sqrt 5$&117093& 2+2\\\cline{3-5}
    &&$301\pm 25\sqrt 5$&117099& 8+8\\\cline{3-5}
    % -326*E(5)-276*E(5)^2-276*E(5)^3-326*E(5)^4,
    &&$301\pm 25\sqrt 5$&117101& 2+2\\\cline{3-5}
  && 424& 117151& 4\\\cline{3-5}
  && 424& 117153& 1\\\cline{3-5}
  && 426& 117159& 4\\\cline{3-5}
  && 426& 117161& 1\\\hline\hline
  %   \multirow{12}{*}{10d}&\multirow{12}{*}{$C_{10}\times C_5$}
  %                        &$299-25\sqrt 5$&117091& 8\\\cline{3-5}
  %                        % -324*E(5)-274*E(5)^2-274*E(5)^3-324*E(5)^4
  %   &&$299-25\sqrt 5$&117093& 2\\\cline{3-5}
  %   % -274*E(5)-324*E(5)^2-324*E(5)^3-274*E(5)^4,
  %   &&$299+25\sqrt 5$&117091& 8\\\cline{3-5}
  %   % -274*E(5)-324*E(5)^2-324*E(5)^3-274*E(5)^4,
  %   &&$299+25\sqrt 5$&117093& 2\\\cline{3-5}
  %   % -326*E(5)-276*E(5)^2-276*E(5)^3-326*E(5)^4,
  %   &&$301-25\sqrt 5$&117099& 8\\\cline{3-5}
  %   % -326*E(5)-276*E(5)^2-276*E(5)^3-326*E(5)^4,
  %   &&$301-25\sqrt 5$&117101& 2\\\cline{3-5}
  %   % -276*E(5)-326*E(5)^2-326*E(5)^3-276*E(5)^4,
  %   &&$301+25\sqrt 5$&117099& 8\\\cline{3-5}
  %   % -276*E(5)-326*E(5)^2-326*E(5)^3-276*E(5)^4,
  %   &&$301+25\sqrt 5$&117101& 2\\\cline{3-5}
  % % rational 10
  % && 424& 117151& 4\\\cline{3-5}
  % && 424& 117153& 1\\\cline{3-5}
  % && 426& 117159& 4\\\cline{3-5}
  % && 426& 117161& 1\\\hline
%
  % -675*E(5)-650*E(5)^2-650*E(5)^3-675*E(5)^4,
    \multirow{2}{*}{25a}&\multirow{2}{*}{$C_{25}$}&$(1325\pm 25\sqrt 5)/2$& 237342& 10\\\cline{3-5}
    % \multirow{3}{*}{25a}&\multirow{3}{*}{$C_{25}$}&$(1325-25\sqrt 5)/2$& 237342& 10\\\cline{3-5}
    % % -650*E(5)-675*E(5)^2-675*E(5)^3-650*E(5)^4,
    % &&$(1325-25\sqrt 5)/2$&237342& 10\\\cline{3-5}
  && 725& 237352& 5\\\hline
  \end{tabular}
  
  \caption{Values of $\nu_5(g,\eta)$ and $\nu_6(g,\eta)$ for $\eta\in\Irr(C_G(g))$ and $g$ from four conjugacy classes of $G_2(G)$; all other classes lead to only rational indicators.}
  \label{fig:indival}

\end{table}

If we had not listed the sixth indicators, classes 10c and 10d would look identical. Thus, certain representations that share the same fifth indicators are distinguished by their sixth indicators. The choice of $\nu_6$ is not completely arbitrary (though theoretically unfounded): The indicators associated to class 10d that share the same fifth indicator also have the same $\nu_m$ for any $m|\exp(G)$ that is a prime power (that is, $m\in\{2,4,8,3,5,25,7,31\}$). We also note that, at least for the classes shown, two objects that have the same fifth and sixth indicators also share identical values for all the other higher indicators; i.e.\ the listed ``multiplicity'' is the multiplicity with which an entire indicator \emph{sequence} occurs in this class.

The four classes listed in \cref{fig:indival} are the only ones (among $44$) for which irrational indicators occur. The irrationality test \lstinline!FSZtest! above, on the other hand, does not compute indicators, but rather the values $\beta_m(z,\chi)$ for characters $\chi$ of $C_G(z)$. The unique conjugacy class $z^{G_2(5)}$ for which an irrational value $\beta_m(z,\chi)$ occurs is the one labeled 5a. For the various characters of $C_{G_2(5)}(z)$ we get the following values of $\beta_5(z,\chi)$: $0$ occurs $58$ times, $150$ and $(175\pm 25\sqrt 5)/2$ occur twice each, and $450\pm 150\sqrt 5$, $300\pm 50\sqrt 5$, $25$, $100$, and $600$ occur once each. We have $|C_{G_2(5)}|=375,000=2^3\cdot 3\cdot 5^6$; the first counterexample to the $FSZ$-property found in \cite{MR3275646} is for a group of order $5^6$. In fact, one can (GAP) check that the 5-Sylow Subgroups of $G_2(5)$ are not $FSZ$. One might suspect that the $FSZ_p$ property of a group might be related to the $FSZ_p$ property for its $p$-subgroups, but perhaps the scarcity of known non-$FSZ$ groups makes this a quite speculative question; we can merely prove the rather weak:
\begin{Lem}
  Let $q$ be a power of a prime $p$ and $G$ a finite group.
  \begin{enumerate}
  \item If $G$ is not $FSZ_{q}$, then there is a $p$-element $z\in G$ such that $C_G(z)$ is not $FSZ_q$.
  \item If $y\in Z(G)$, and $y_p$ denotes the $p$-part of $y$ then $\beta_q(y,\chi)=\beta_q(y_p,\chi)$.
  \end{enumerate}
\end{Lem}
\begin{proof}
  Let $y$ be such that $\beta_q(y,\chi)=\beta_q^G(y,\chi)\not\in\mathbb Q$. Let $z$ be the $p$-part of $y$. Since $C_{C_G(z)}(y)=C_G(y)$, we have $\beta_q^{C_G(z)}(y,\chi)=\beta_q^G(y,\chi)$, and thus $C_G(z)$ is not $FSZ$.

  If $y\in Z(G)$, then $y_p$ as well as the $p'$-part $y_{p'}$ of $y$ are central. Let
  $\exp(G)=p^kn$ with $(n,p)=1$, and $qr\equiv 1\mod n$. Let $x\in G$ with $p$-part $x_p$ and $p'$-part $x_{p'}$. Then $x^q=z$ is equivalent to $x_p^q=z_p$ and $x_{p'}^q=y_{p'}$. Since $x_{p'}$ and $y_{p'}$ are $p'$-elements, $x_{p'}^q=y_{p'}$ implies  $x_{p'}=y_{p'}^n=:a$; on the other hand $a$ is central in $G$ and thus $(ab)^q=y$ is equivalent to $b^q=y_p$ for $p$-elements $b$. Therefore
  \begin{multline*}
    \langle\chi,w_q^y\rangle=\sum_{x^q=y}\chi(x)
                            =\sum_{b^q=y_p}\chi(ab)
                            =\sum_{b^q=y_p}\frac{\chi(a)}{\chi(e)}\chi(b)
    =\frac{\chi(a)}{\chi(e)}\langle\chi,w_q^{y_p}\rangle,
  \end{multline*}
  where $\chi(a)/\chi(e)$ is a root of unity.
\end{proof}
The lemma implies that a non-$FSZ_q$-group $G$ has a $p$-element $z$ such that $C_G(z)$ has a central $p$-element $y$ such that $\beta_q^{C_G(z)}(y,\chi)$ is irrational for some $\chi\in\Irr(C_G(z))$.

\printbibliography

\end{document}